\documentclass[a4paper,notitlepage,twoside,reqno,12pt]{amsart}
\usepackage{CJK}
\usepackage{mathrsfs}
\usepackage{bbm,pifont,latexsym}
\usepackage{anysize}
\marginsize{2.4cm}{2.4cm}{2.4cm}{2.4cm}
\usepackage{CJK}
\usepackage{dcolumn,indentfirst,color}
\usepackage[pdfpagemode=UseNone,pdfstartview=FitH]{hyperref}
\usepackage{amsmath,amssymb,amscd,amsthm,amsfonts,mathrsfs}
\usepackage{color,graphicx,xcolor,graphics,enumerate,comment}
\usepackage{epstopdf}
\usepackage{titlesec}
\usepackage{titletoc}
\usepackage{multirow}
\titlecontents{section}[20pt]{\addvspace{2pt}\filright}
              {\contentspush{\thecontentslabel. }}
              {}{\titlerule*[8pt]{}\contentspage}%

\usepackage{fancyhdr} %
\makeatletter
\@namedef{subjclassname@2020}{%
  \textup{2020} Mathematics Subject Classification}
\makeatother

\newtheorem{thm}{Theorem}[section]
\newtheorem{lem}[thm]{Lemma}
\newtheorem{cor}[thm]{Corollary}
\newtheorem{prop}[thm]{Proposition}
\newtheorem{rem}[thm]{Remark}
\newtheorem{rmk}[thm]{Remark}
\newtheorem{defi}[thm]{Definition}

\theoremstyle{definition}
\newtheorem*{exam}{Example}

\newcommand{\BD}[1]{\mathbf{#1}}
\newcommand{\bd}[1]{\mathbf{#1}}
\newcommand{\ud}{\mathrm{d}}
\newcommand{\norm}[1]{\left\lVert#1\right\rVert}
\newcommand{\N}{\mathbb{N}}
\newcommand{\T}{\mathbb{T}}

\newcommand{\R}{\mathbb{R}}
\newcommand{\Z}{\mathbb{Z}}
\newcommand{\Q}{\mathbb{Q}}
\newcommand{\MA}{\mathcal{A}}

\newcommand{\MG}{\mathcal{G}}
\newcommand{\MS}{\mathcal{S}}

\newcommand{\MM}{\mathcal{M}}
\newcommand{\MK}{\mathcal{K}}
\newcommand{\MH}{\mathcal{H}}

\makeatletter\@addtoreset{equation}{section}\makeatother %

\titleformat{\section}{\centering\normalsize}{\textsc{\thesection.}}{0.5em}{\textsc}
\titleformat{\subsection}[runin]{\normalsize}{\textbf{\thesubsection.}}{0.3em}{\textbf}

\begin{document}
\begin{CJK*}{GBK}{song}
\author{Wen-Long Li}
\address{Wen-Long Li, School of Mathematics, Sun Yat-Sen University, Guangzhou, 510275, P. R. China}
\email{liwenlong@mail.sysu.edu.cn}



\title{Mountain pass type solutions for a generalized Frenkel-Kontorova model}

\begin{abstract}
We study a generalized Frenkel-Kontorova model and
obtain periodic and heteroclinic mountain pass solutions.
Heteroclinic mountain pass solution in the second laminations is new to the generalized Frenkel-Kontorova model.
Our proof follows that of Bolotin and Rabinowitz for an Allen-Cahn equation,
which is different with heat flow method
for finding critical point of Frenkel-Kontorova model
in the literature.
The proofs depend on suitable choices of functionals and working spaces.
We also study the multiplicity of these mountain pass solutions.
\end{abstract}



\date{\today}

\maketitle


\section{Introduction}\label{sec:intro}

In this paper, we study a generalized Frenkel-Kontorova (or FK) model.
To introduce the FK model, we need some notations.
Let $\bd{i},\bd{j},\bd{k}$ (resp. $i,j,k$), etc. denote the elements of $\Z^n$ (resp. $\Z$)
and define $\norm{\bd{i}}:=\sum_{j=1}^{n}|\bd{i}_j|$.
Fix $r\in\N$ and set $B_{\bd{0}}^{r}=\{\bd{k}\in\Z^n \,|\, \norm{\bd{k}}\leq r\}$.
Assume that $s\in C^2(\R^{B_{\BD{0}}^{r}}, \R)$ satisfies (cf. \cite{mramor}):
\begin{enumerate}[({S}1)]
  \item \label{eq:S1} $s(u+1_{B_{\BD{0}}^{r}})=s(u)$, where $1_{B_{\BD{0}}^{r}}$ is the constant function $1$ on $B_{\BD{0}}^{r}$;
  \item \label{eq:S2} $s$ is bounded from below and coercive in the following sence,
      \begin{equation*}
        \lim_{|u(\BD{k})-u(\BD{j})|\to \infty}s(u)=\infty, \textrm{ for $\BD{k}, \BD{j}\in B_{\BD{0}}^{r}$ with $\norm{\BD{k}-\BD{j}}=1$;}
      \end{equation*}

  \item \label{eq:S3} $\partial_{\BD{k}, \BD{j}}s \leq 0$ for $\BD{k}, \BD{j}\in B_{\BD{0}}^{r}$ with $\BD{k}\neq \BD{j}$, while $\partial_{\BD{0},\BD{j}}s<0$ for $\norm{\BD{j}}=1$;
  \item \label{eq:S4} there is some constant $C$ such that $|\partial_{\bd{i},\bd{k}}s|\leq C$ for all $\bd{i}, \bd{k}\in B_{\BD{0}}^{r}$.
\end{enumerate}
For $u\in\R^{\Z^n}$, set
\begin{equation*}
  S_{\BD{j}}(u)=s(\tau_{-\BD{j}_n}^{n}\cdots \tau_{-\BD{j}_1}^{1}u|_{B_{\BD{0}}^{r}}),
\end{equation*}
where $\tau_{-k}^{j}: \R^{\Z^n}\to \R^{\Z^n}$ is defined by $(\tau_{-k}^{j}u)(\BD{i})=u(\BD{i}+k\BD{e}_j)$ with $\bd{e}_j=(0,\cdots,1,\cdots,0)$, i.e., the $j$th component is $1$ and others $0$.
With these local potentials $S_{\BD{j}}$, we can define the formal sum
\begin{equation*}
\sum_{\BD{j}\in\Z^n}S_{\BD{j}}(u)
\end{equation*}
and its Euler-Lagrange equation
\begin{equation}\label{eq:PDE}
    \sum_{\BD{j}\in\Z^n}\partial_{\BD{i}}S_{\BD{j}}(u)=\sum_{\BD{j}:\norm{\BD{j}-\BD{i}}\leq r}\partial_{\BD{i}}S_{\BD{j}}(u)=0, \quad\textrm{ for all } \BD{i}\in\Z^n.
\end{equation}
\eqref{eq:PDE} is the equation of our generalized FK moldel.

FK model was first proposed in 1938 (\cite{FK}), since then it
``has become one of the fundamental and
universal tools of low-dimensional nonlinear physics'' (\cite[p. VII, line 16]{Braun}).
FK model is constituted by a chain of atoms subjected to a periodic potential
and is described by the following
equation:
\begin{equation*}
  \frac{\ud ^2 u}{\ud t^2}(i)-[u(i+1)+u(i-1)-2u(i)]+V'(u(i))=0, \quad \textrm{for all $i\in\Z$.}
\end{equation*}
Here $V\in C^2(\R,\R)$ is $1$-periodic.
Equilibrium or stationary state of FK model is a function $u:\Z\to \R$ satisfying
\begin{equation}\label{eq:1.2al}
  -[u(i+1)+u(i-1)-2u(i)]+V'(u(i))=0,\quad \textrm{for all $i\in\Z$.}
\end{equation}
Our generalized FK model \eqref{eq:PDE} is a generalization of \eqref{eq:1.2al} by setting
\begin{equation*}
s(u|_{B_{0}^{1}})=\frac{1}{8}\left\{[u(1)-u(0)]^2+[u(-1)-u(0)]^2\right\}+V(u(0))
\end{equation*}
and $S_{j}(u)=s(\sigma_{-j}u|_{B_{0}^{1}})$, where $(\sigma_{-j}u)(\cdot)=u(\cdot+j)$.
So solutions of \eqref{eq:PDE} are also called equilibrium or stationary state of the generalized FK model.

In 1983, Aubry and Le Daeron (\cite{Aubry}) studied minimal solutions of \eqref{eq:1.2al} and obtained the classification of minimal solutions.
Minimal solutions are one of the important classes of equilibrium state.
For \eqref{eq:1.2al}, a function $u$ is said to be \emph{minimal} if
\begin{equation}\label{eq:minimal888}
  \sum_{j\in \Z}\left(S_{j}(u+v)-S_{j}(u)\right)\geq 0
\end{equation}
for any $v$ with $\{i\in\Z \,|\,v(i)\neq 0\}$ a finite set.
Aubry and Le Daeron found that minimal solution $u$ did not cross with any of its translation $u(\cdot-j)+l$,
which led to an
oriented homeomorphism map of a circle and then
a rotation number.
Using rotation number Aubry and Le Daeron made the classification of minimal solutions.
Now their results are called Aubry-Mather theory because Mather (\cite{Mather}) obtained similar results for monotone twist maps of annulus.

After the establishment of Aubry-Mather theory, Moser \cite{Moser} attempted to generalize this theory to elliptic PDE.
He found that for higher dimensional space, minimal solution might cross with its translation.
So he posed another property, i.e., without self-intersections on minimal solution.
In other words, Moser asked $u$ satisfied one and only one of the following inequality holds:
\begin{equation}\label{eq:moser8588}
  u(x-j\bd{e}_k)+l>u, \quad\textrm{or  }\,\, u(x-j\bd{e}_k)+l=u, \quad\textrm{or  }\,\, u(x-j\bd{e}_k)+l<u.
\end{equation}
Moser and then Bangert (\cite{Bangert}) studied a class of elliptic PDE and they obtained similar results of Aubry-Mather theory.
Now their results are called Moser-Bangert theory (\cite{RS}).
Bolotin, Rabinowitz, Stredulinsky (\cite{RS, Rabi2006, Rabi2007, Rabi2014}) studied an Allen-Cahn equation, which belonged to the elliptic PDE of Moser and Bangert.
They used variational methods to construct more homoclinic and heteroclinic solutions of the Allen-Cahn equation other than Moser and Bangert's.

In \cite{mramor, Miao, LC}, Birkhoff minimizers (corresponding to minimal and without self-intersections solutions in Moser-Bangert theory) have been established and multitransition solutions was constructed in \cite{LC1}.
In this paper, we shall use the methods of \cite{Rabi2006, Rabi2007} to establish a new type of solution, mountain pass solution.
Noting that in \cite{mramor, Miao, LC, LC1}, \eqref{eq:PDE} was studied without the assumption (S\ref{eq:S4}) except in \cite{mramor}.
Our results can be seen as a new proof and a refinement of some results of \cite{mramor} (see also \cite{dev20071, dev2007}).
Note that we only consider the case of rotation vector $\alpha\in\Q^n$.
In \cite{mramor}, Mountain Pass Theorem was also used to establish critical point.
But to prove Mountain Pass Theorem (cf. \cite[Lemma 8.6]{mramor}),
Mramor and Rink asked the functional to be a Morse function.
If the functional is a Morse function, they obtained a ghost circle which contained a periodic mountain pass solution.
When the functional is not a Morse function,
using a limiting progress,
Mramor and Rink established a ghost circle that contained a stationary solution.
If a gap of periodic minimal and Birkhoff solutions
is not filled up by minimal solutions,
the above stationary solution should be not minimal.
The proofs of this paper are more direct than that of \cite{mramor}.
We also establish heteroclinic
mountain pass solution in the second laminations (please see \cite{Miao} for the definition of second laminations)
while
Mramor and Rink's result only holds for the ``first'' lamination.

But we point out that in \cite{mramor}, the authors obtained non-minimal solution
for rotation vector $\alpha\in\R^n\setminus \Q^n$ such that the Aubry-Mather set had gap, provided that ghost circle was not consists of minimizers.
In \cite{dev20071, dev2007}, the authors showed that there was some critical point in the gap of ground states of some FK model
for any rotation vector $\alpha\in\R^n$ such that gap (in the ``first'' lamination) condition held.
Our result does not cover these cases and
we limit ourselves in the case that $\alpha\in\Q^n$.
We also prove the multiplicity of mountain pass solutions which is new to this generalized FK model.
%
Other FK type models (cf. \cite{dev20071, dev2007} and references there in)
may be studied using the method of the present paper
and will be considered in the future.

This paper is organized as follows. We introduce some definitions and lemmas in
Section \ref{sec:pre}.
In Section \ref{sec:mountainpass}, periodic mountain pass solution is established and
it is proved that there are infinitely many solutions of this type.
Heteroclinic mountain pass solution is considered and the multiplicity is studied
in Section \ref{sec:2006}.
In Appendix \ref{sec:app1}, we present the detailed proofs of some properties of Section \ref{sec:mountainpass}.
A heat flow method for proving the existence of mountain pass solution is also included in Appendix \ref{sec:app1}.

\section{Preliminary}\label{sec:pre}

We review some definitions and some lemmas of \cite{mramor, Miao, LC, LC1}.
Assume $s$ satisfies (S\ref{eq:S1})-(S\ref{eq:S3}) in this section.
For functions $u,v\in\R^{\Z^n}$,
$v<u$ means $v(\bd{i})<u(\bd{i})$ for all $\bd{i}\in \Z^n$, and similarly one define $=, >, \geq, \leq,$ etc.
The following lemmas provide important comparison results.
\begin{lem}[{cf. \cite[Lemma 2.6]{Miao}}]\label{lem:2.6miao}
For $u, v\in\R^{\Z^n}$ and an arbitrary finite set $B\subset \Z^n$, we have
\begin{equation*}
    \sum_{\bd{j}\in B}S_{\bd{j}}(\max(u,v))+\sum_{\bd{j}\in B}S_{\bd{j}}(\min(u,v))\leq
    \sum_{\bd{j}\in B}S_{\bd{j}}(u)+\sum_{\bd{j}\in B}S_{\bd{j}}(v).
\end{equation*}
\end{lem}
\begin{lem}[{cf. \cite[(3.1)]{LC}}]\label{lem:subsequ}
If $\{u_n\}_{n\in\N}$ satisfies $v\leq u_n\leq w$ for fixed $v,w\in \R^{\Z^n}$,
then there is a subsequence of $\{u_n\}_{n\in\N}$ converging pointwise.
\end{lem}
\begin{lem}[{cf. \cite[Lemma 2.5]{Miao}; \cite[Lemma 4.5]{mramor}}]\label{lem:2.5miao}
Assume that $u$ and $v$ are solutions of \eqref{eq:PDE} and $u\leq v$. Then either $u<v$ or $u=v$.
\end{lem}
A function $u$ is said to \emph{have bounded action} if there exists $C>0$,
such that $|u(\bd{k})-u(\bd{j})|\leq C$ for all $\bd{k},\bd{j}\in\Z^n$ with $\norm{\bd{k}-\bd{j}}=1$ (cf. \cite[p. 1525, line -3]{Miao}, \cite[p. 1112, line -8]{LC}).
\begin{lem}[{cf. \cite[Lemma 2.4]{Miao}, \cite[Lemma 2.11]{LC}}]\label{lem:lc16321}
Assume $u,v\in\R^{\Z^n}$ have bounded action with bounded constant $C$.
Then there exists a constant $L=L(C,r)>0$ such that for any finite set $B\subset \Z^n$,
\begin{equation*}
 | \sum_{\bd{j}\in B}S_{\bd{j}}(u)-\sum_{\bd{j}\in B}S_{\bd{j}}(v)|\leq L\sum_{\bd{j}\in\bar{B}}|(u-v)(\bd{j})|.
\end{equation*}
Here the closure of a set $B$ is defined by $\bar{B}=\cup_{\bd{j}\in B}\{\bd{k}\in\Z^n\,|\, \norm{\bd{k}-\bd{j}}\leq r\}$.
\end{lem}

Similar to \eqref{eq:minimal888} and \eqref{eq:moser8588}, we introduce the following definition.
\begin{defi}\label{defi:2.3mramor}
\begin{itemize}
  \item (cf. \cite[Definition 2.3]{mramor}) A function $u:\Z^n \to \R$ is said to be minimal for potentials $S_{\BD{j}}$ (or for potential $s$) if for every finite subset $B\subset \Z^n$ and every $v:\Z^n\to \R$ with support, denoted by $supp (v)$, included in $int_r(B)$,
\begin{equation*}
    \sum_{\BD{j}\in B}(S_{\BD{j}}(u+v)-S_{\BD{j}}(u))\geq 0,
\end{equation*}
where the support of $v$ is $supp (v):=\{\BD{i}\in\Z^n\,|\, v(\BD{i})\neq 0\}$ and interior of $B$ is
$int_r(B)=\{\BD{i}\in B\,|\, \bd{i}+B_{\bd{0}}^{r}\subset B\}$.
  \item (cf. \cite[Definition 2.1]{Miao}, \cite[p.3, line 25]{RS}) A function $u$ is said to be Birkhoff if $\{ \tau_{j}^{k}u\,|\, j\in\Z \textrm{ and } 1\leq k \leq n \}$ is totally ordered, i.e., for all $j\in\Z$ and $1\leq k \leq n$, it follows that
\begin{equation*}
    \tau_{j}^{k}u<u, \textrm{\quad or \quad} \tau_{j}^{k}u=u, \textrm{\quad or \quad} \tau_{j}^{k}u>u.
\end{equation*}
\end{itemize}
\end{defi}
For $\bd{p}=(\bd{p}_1,\cdots, \bd{p}_n)\in\N^n$, let
\begin{equation*}
\begin{split}
  &\R^{\Z^{n}/(\bd{p}\Z^n)}\\
  :=&\{u:\Z^{n}\to \R \,|\, u(\bd{i}+\bd{p}_j \cdot\bd{e}_j)=u(\bd{i}),\quad \textrm{for any $j
  \in \{1,\cdots, n\}$ and $\bd{i}\in\Z^n$}\}.
  \end{split}
\end{equation*}
If $\bd{p}=(1,\cdots,1)\in \N^n$, we use $\R^{\Z^n /\Z^n}$ to replace $\R^{\Z^n /(\bd{p}\Z^n)}.$
Similarly for $\bd{q}=(\bd{q}_2,\cdots, \bd{q}_n)\in\N^{n-1}$, one define $\R^{\Z\times \Z^{n-1}/(\bd{q}\Z^{n-1})}$, which
consists of functions that is periodic in $\bd{i}_2, \cdots, \bd{i}_n$ with periods $\bd{q}_2, \cdots, \bd{q}_n$.

\subsection{Periodic minimal and Birkhoff solutions}\label{sec:4524573}
\

For $u\in\R^{\Z^n/\Z^n}$,
define $J_0 (u):=S_{\bd{0}}(u)$,
$c_0:=\inf_{u\in \R^{\Z^n/\Z^n}}J_0(u)$ and $\MM_0:=\{u\in\R^{\Z^n/\Z^n}\,|\, J_0 (u)=c_0\}$.
It was proved in \cite{LC} that $\MM_0(\neq \emptyset)$ was ordered and
consisted of minimal and Birkhoff solutions of \eqref{eq:PDE}.
Replacing $\R^{\Z^n/\Z^n}$ by $\R^{\Z^n/(\bd{p}\Z^n)}$
and minimizing the corresponding functional, we do not obtain more periodic solutions, as stated in the following.

For $\bd{p}=(\bd{p}_1,\cdots,\bd{p}_n)\in\N^n$,
let
$$\T^{\bd{p}}_{0}:=\{0,\cdots,\bd{p}_1-1\}\times\{0,\cdots,\bd{p}_2-1\}\times \cdots \times\{0,\cdots,\bd{p}_n-1\}.$$
and $\Gamma_{0}^{\bd{p}}:=\R^{\Z^{n}/(\bd{p}\Z^n)}$.
For $u\in \Gamma_0^{\bd{p}}$, define
\begin{equation}\label{eq:1.2}
  J^{\bd{p}}_{0}(u):=\sum_{\bd{j}\in\T^{\bd{p}}_{0}}S_{\bd{j}}(u).
\end{equation}
The following lemma was proved in \cite{LC} by Moser's method (cf. \cite{Moser}, see also \cite[Proposition 2.2]{RS}).
\begin{lem}[{cf. \cite[Proposition 3.1]{LC}}]\label{prop:2.2369852145}
Let $\BD{p}\in \N ^n$ and $
c_0 ^{\BD{p}}:=\inf_{u\in\Gamma_{0}^{\bd{p}}}J_{0}^{\bd{p}}(u).
$
Then $
\MM_0 ^{\BD{p}}:=\{u\in \Gamma_0 ^{\BD{p}}\,|\,J_{0}^{\bd{p}}(u)=c_0 ^{\BD{p}}\}\neq \emptyset.
$
Moreover, $\MM_0 ^{\BD{p}}=\MM_{0} $ and $c_0 ^{\BD{p}}=(\prod_{i=1}^n \BD{p}_i)c_{0}$.
\end{lem}

Suppose that $\MM_0 $
constitutes a lamination, or in other words,
there is a gap in $\MM_0$, i.e.,
\begin{equation}\label{eq:*0}
\textrm{there are $v_0,w_0\in \mathcal{M}_0$ with $v_0<w_0$ such that $v_0, w_0$ are adjacent}.\tag{$*_0$}
\end{equation}
Here \emph{adjacent} means there does not exist $u\in\MM_0$ such that $v_0\leq u\leq w_0$.
In \cite{LC}, heteroclinic minimal and Birkhoff solutions are constructed under condition \eqref{eq:*0}.

\subsection{Heteroclinic minimal and Birkhoff solutions in $\bd{i}_1$}\label{sec:312432}
\

To construct heteroclinic minimal and Birkhoff solutions, assume that \eqref{eq:*0} holds.
Let $\bd{T}_i=i\bd{e}_1$. Set
$\hat{\Gamma}_1(v_0,w_0):=\{u\in \R^{\Z\times \Z^{n-1}/\Z^{n-1}}\,|\,v_0\leq u\leq w_0\}$.
For $u\in \hat{\Gamma}_1(v_0,w_0)$, define
$J_{1;p,q}(u):=\sum_{i=p}^{q}[J_0(\tau_{-i}^{1}u)-c_0]$,
then it was proved in \cite[Proposition 3.2]{LC} that $J_{1;p,q}(u)\geq -K_1$ for some $K_{1}=K_{1}(v_0,w_0)\geq 0$.
Thus we can define
\begin{equation}\label{eq:haha1}
J_{1}(u):=\liminf_{p\to -\infty\atop q\to \infty}J_{1;p,q}(u),
\end{equation}
and we have (by \cite[Lemma 3.3]{LC})
\begin{equation}\label{eq:pppp5}
  J_{1;p,q}(u)\leq J_1(u)+2K_1.
\end{equation}
Set
\begin{equation*}
\begin{split}
  \Gamma_1(v_0,w_0)
  :=\{u\in\hat{\Gamma}_1(v_0,w_0)\,|\,  &\lim_{i\to -\infty}|(u-v_0)(\bd{T}_i)|=0,\\
  &\lim_{i\to \infty}|(u-w_0)(\bd{T}_i)|=0\}.
  \end{split}
\end{equation*}
For $u\in \Gamma_1(v_0,w_0)$, as was proved in \cite[Proposition 3.4]{LC}, if $J_1(u)<\infty$, then
\begin{equation*}
J_1(u)=\lim_{p\to -\infty\atop q\to \infty}J_{1;p,q}(u), \quad \textrm{i.e., } J_1(u)=\sum_{i\in\Z}[J_0(\tau_{-i}^{1}u)-c_0].
\end{equation*}
In other words, $\liminf$ becomes $\lim$ in the definition of $J_1(u)$.
Set $$c_1(v_0,w_0):=\inf_{u\in\Gamma_1(v_0,w_0)}J_1(u).$$
Then, as was proved in \cite{LC}, $c_1(v_0,w_0)$ is attained
and $$\MM_1(v_0,w_0):=\{u\in\Gamma_1(v_0,w_0)\,|\, J_1(u)=c_1(v_0,w_0)\}$$ is an ordered set and consists of heteroclinic minimal and Birkhoff solutions of \eqref{eq:PDE}.
Moreover, we have
\begin{lem}[{cf. \cite[Proposition 2.13]{LC1}}]\label{lem:6553}
Suppose \eqref{eq:*0} holds and $u\in\hat{\Gamma}_1(v_0,w_0)$ with $J_1(u)<\infty$.
If $u$ satisfies \eqref{eq:PDE}
for $\BD{i}_1\geq R$ (resp. $\BD{i}_1\leq -R$), then
$|(u-\phi)(\bd{T}_i)|\to 0$ as $i\to \infty$ (resp. $|(u-\phi)(\bd{T}_i)|\to 0$ as $i\to -\infty)$,
where $R\in\N$ and $\phi=v_0$ or $w_0$.
\end{lem}

Similar to Section \ref{sec:4524573},
varying the periods of function in $\Gamma_1(v_0,w_0)$ cannot produce more minimal and Birkhoff solution.
To see this,
for $\bd{q}=(\bd{q}_2,\cdots, \bd{q}_{n})\in\N^{n-1}$ let $$\T_{1}^{\bd{q}}:=\{0,\cdots, \bd{q}_2 -1\}\times\cdots\times \{0,\cdots, \bd{q}_{n}-1 \}.$$
Set
\begin{equation}\label{eq:gamma186541100}
\begin{split}
\Gamma_{1}^{\bd{q}}(v_0,w_0)
:=\{u\in\R^{\Z\times \Z^{n-1} /(\bd{q}\Z^{n-1})}\,|\, &v_0\leq u\leq w_0, \\
&\lim_{i\to - \infty}\sum_{\bd{j}\in \{i\}\times\T_{1}^{\bd{q}}}|(u-v_0)(\bd{j})|=0, \\
&\lim_{i\to \infty}\sum_{\bd{j}\in \{i\}\times\T_{1}^{\bd{q}}}|(u-w_0)(\bd{j})|=0
\}.
\end{split}
\end{equation}
For $u\in \Gamma_{1}^{\bd{q}}(v_0,w_0)$, define
\begin{equation*}
  J_{1;p,q}^{\bd{q}}(u):=\sum_{i=p}^{q}J_{1,i}^{\bd{q}}(u):=\sum_{i=p}^{q} \sum_{\bd{j}\in \{i\}\times \T_{1}^{\bd{q}} }\left[S_{\bd{j}}(u)-c_0 \right],
\end{equation*}
and
\begin{equation*}
J_{1}^{\bd{q}}(u):=\liminf_{p\to -\infty \atop q\to \infty}J_{1;p,q}^{\bd{q}}(u).
\end{equation*}
Similar to \eqref{eq:haha1} $J_{1}^{\bd{q}}(u)$ is well-defined and it satisfies
\begin{lem}[{cf. \cite[Proposition 3.4]{LC}}]\label{prop:8554556}
For $u\in \Gamma_{1}^{\bd{q}}(v_0,w_0)$, if $J_{1}^{\bd{q}}(u)<\infty$,
\begin{equation*}
J_{1}^{\bd{q}}(u)=\lim_{p\to -\infty \atop q\to \infty}J_{1;p,q}^{\bd{q}}(u), \quad \textrm{i.e., }J_{1}^{\bd{q}}(u)=\sum_{i\in \Z}J_{1,i}^{\bd{q}}(u).
\end{equation*}
\end{lem}
\begin{rem}\label{rem:000123123}
Suppose that $v,w\in \MM_0$ satisfy $v\leq v_0 <w_0\leq w$ and $v,w$
may be not adjacent.
A careful reading of the proof of \cite[Proposition 3.4]{LC} tells us that
for %
\begin{equation*}%
\begin{split}
u\in\{u\in\R^{\Z\times \Z^{n-1} /(\bd{q}\Z^{n-1})}\,|\, &v\leq u\leq w, \\
&\lim_{i\to - \infty}\sum_{\bd{j}\in \{i\}\times\T_{1}^{\bd{q}}}|(u-v_0)(\bd{j})|=0, \\
&\lim_{i\to \infty}\sum_{\bd{j}\in \{i\}\times\T_{1}^{\bd{q}}}|(u-w_0)(\bd{j})|=0
\},
\end{split}
\end{equation*}
$J_{1}^{\bd{q}}(u)$ is well-defined and Lemma \ref{prop:8554556} holds.
\end{rem}

Similar to Lemma \ref{prop:2.2369852145}, we have
\begin{lem}[{cf. \cite[Proposition 3.20]{LC}}]\label{prop:2.235655145}
Let $\BD{q}=(\BD{q}_2,\cdots, \BD{q}_n)\in \N ^{n-1}$ and $c_{1}^{\bd{q}}(v_0,w_0):=\inf_{u\in \Gamma_{1}^{\bd{q}}(v_0,w_0)}J_{1}^{\bd{q}}(u)$.
Then $ 
\MM_1 ^{\BD{q}}:=\{u\in \Gamma_1 ^{\BD{q}}(v_0,w_0)\,|\,J_{1}^{\bd{q}}(u)=c_1 ^{\BD{q}}(v_0,w_0)\}\neq \emptyset.
$ 
Moreover, $\MM_{1}^{\bd{q}}(v_0,w_0)=\MM_1(v_0,w_0)$, and $c_{1}^{\bd{q}}(v_0,w_0)=(\prod_{i=2}^n \BD{q}_i)c_1(v_0,w_0)$.
\end{lem}
In analogy with \eqref{eq:*0},
assume
\begin{equation}\label{eq:*1}
    \textrm{there are } v_1, w_1 \in \MM_1(v_0,w_0) \textrm{ with } v_1<w_1 \textrm{ such that $v_1, w_1$ are adjacent}\tag{$*_1$}.
\end{equation}
In Section \ref{sec:mountainpass} we shall establish the existence of periodic mountain pass solution in the gap of $v_0,w_0$ while
in Section \ref{sec:2006}, we shall construct heteroclinic mountain pass solution in the gap of $v_1,w_1$.

\section{Mountain pass solutions in the gap of $\MM_0$}\label{sec:mountainpass}

Assume that $s$ satisfies (S\ref{eq:S1})-(S\ref{eq:S4}) in this and the following two sections.
We establish periodic mountain pass solution of \eqref{eq:PDE} in this section.
Firstly we introduce the working space and the corresponding functional.
For $\bd{p}\in\N^n$, set
\begin{equation*}
  \Lambda_{0}^{\bd{p}}:=\left\{u\in\R^{\Z^{n}/(\bd{p}\Z^n)}\,\Big|\,\norm{u}^{2}_{\Lambda_{0}^{\bd{p}}}:=\sum_{\bd{j}\in \T_{0}^{\bd{p}}}|u(\bd{j})|^2 <\infty\right\}.
\end{equation*}
It is easy to see that $(\Lambda_{0}^{\bd{p}}, \norm{\cdot}_{\Lambda_{0}^{\bd{p}}})$ is a Banach space.
Define $J^{\bd{p}}_{0}$ as in Section \ref{sec:4524573} and assume that \eqref{eq:*0} holds.
For $u\in \Lambda_{0}^{\bd{p}}$,
set $I^{\bd{p}}_{0}(u):=J^{\bd{p}}_{0}(u+v_0)$.
Then since $s\in C^2(\R^{B_{\BD{0}}^{r}},\R)$, $I^{\bd{p}}_{0}\in C^1(\Lambda_{0}^{\bd{p}},\R)$ and
\begin{equation}\label{eq:derivative}
  \begin{split}
    (I_{0}^{\bd{p}})'(u)v =& \sum_{\bd{j}\in \T^{\bd{p}}_{0}}\sum_{\bd{k}:\norm{\bd{k}-\bd{j}}\leq r}\partial_{\bd{k}}S_{\bd{j}}(u+v_0)v(\bd{k}) \\
      =& \sum_{\bd{j}\in \T^{\bd{p}}_{0}}v(\bd{j})\sum_{\bd{k}:\norm{\bd{k}-\bd{j}}\leq r}\partial_{\bd{k}}S_{\bd{j}}(u+v_0),
  \end{split}
\end{equation}
where $(I_{0}^{\bd{p}})'$ is the Fr\'echet derivative
of $I_{0}^{\bd{p}}$.
If $(I_{0}^{\bd{p}})'(u)=0$, then $$\sum_{\bd{k}:\norm{\bd{k}-\bd{j}}\leq r}\partial_{\bd{k}}S_{\bd{j}}(u+v_0)=0$$ hold for all $\bd{j}\in \T^{\bd{p}}_{0}$.
Hence by the
periodicities of $u$ and $v_0$, $u+v_0$ is a solution of \eqref{eq:PDE}.

\subsection{Periodic mountain pass solution}
\

Consider the semiflow $\Phi_{t}^{0} : \Lambda_{0}^{\bd{p}} \rightarrow \Lambda_{0}^{\bd{p}}$,
which is defined by 
\begin{equation}\label{eq:dela7}
  \left\{
    \begin{array}{ll}
      -\partial_t \Phi_{t}^{0} (u)(\bd{i}) & =\sum_{\bd{j}: \norm{\bd{j}-\bd{i}}\leq r}\partial_{\bd{i}}S_{\bd{j}}(\Phi_{t}^{0} (u)+v_0), \quad \quad \textrm{for }t>0,\\
      \Phi_{0}^{0} (u)(\bd{i}) & =u(\bd{i}).
    \end{array}
  \right.
\end{equation}
Set $W(u)(\bd{i})=\sum_{\bd{j}: \norm{\bd{j}-\bd{i}}\leq r}\partial_{\bd{i}}S_{\bd{j}}(u+v_0)$, then $W(u)\in \Lambda_{0}^{\bd{p}}$ for any $u\in \Lambda_{0}^{\bd{p}}$.
For $u, v\in \Lambda_{0}^{\bd{p}}$,
\begin{equation}\label{eq:lipschitz}
  \begin{split}
     &\norm{W(u)-W(v)}_{\Lambda_{0}^{\bd{p}}}^{2}\\
     =& \sum_{\bd{i}\in \T^{\bd{p}}_{0}}|W(u)(\bd{i})-W(v)(\bd{i})|^{2} \\
      = &  \sum_{\bd{i}\in \T^{\bd{p}}_{0}} \Big|\sum_{\bd{j}:\norm{\bd{j}-\bd{i}}\leq r}[\partial_{\bd{i}}S_{\bd{j}}(u+v_0)-\partial_{\bd{i}}S_{\bd{j}}(v+v_0)]\Big|^{2}\\
      = &  \sum_{\bd{i}\in \T^{\bd{p}}_{0}} \left|\sum_{\bd{j}:\norm{\bd{j}-\bd{i}}\leq r}\left[\int_{0}^{1}\frac{\ud}{\ud t} \partial_{\bd{i}}S_{\bd{j}}(v+t(u-v)+v_0)\ud t\right]\right|^{2}\\
      =& \sum_{\bd{i}\in \T^{\bd{p}}_{0}}\left|
      \sum_{\bd{j}:\norm{\bd{j}-\bd{i}}\leq r \atop \bd{l}:\norm{\bd{l}-\bd{j}}\leq r}
      \int_{0}^{1} \partial_{\bd{i},\bd{l}}S_{\bd{j}}(v+t(u-v)+v_0)\ud t \cdot[u(\bd{l})-v(\bd{l})]
      \right|^{2}\\
      \leq & \sum_{\bd{i}\in \T^{\bd{p}}_{0}}
      \sum_{\bd{j}:\norm{\bd{j}-\bd{i}}\leq r \atop \bd{l}:\norm{\bd{l}-\bd{j}}\leq r}
      \left(
      \int_{0}^{1} \partial_{\bd{i},\bd{l}}S_{\bd{j}}(v+t(u-v)+v_0)\ud t
      \right)^2
      \sum_{\bd{j}:\norm{\bd{j}-\bd{i}}\leq r \atop \bd{l}:\norm{\bd{l}-\bd{j}}\leq r}
      [u(\bd{l})-v(\bd{l})]^2\\
      \leq & C^2 \cdot C(r)\sum_{\bd{i}\in \T^{\bd{p}}_{0}}\sum_{\bd{j}:\norm{\bd{j}-\bd{i}}\leq r \atop \bd{l}:\norm{\bd{l}-\bd{j}}\leq r}
      [u(\bd{l})-v(\bd{l})]^2\\
      \leq & C^2 \cdot C(r) \cdot C_1(r) \norm{u-v}_{\Lambda_{0}^{\bd{p}}}^{2}
       \end{split}
\end{equation}
where
$C$ is the constant in (S\ref{eq:S4})
and
$C(r),C_1(r)$ are constants depending only on $r$.
By Cauchy-Lipschitz-Picard Theorem (please see e.g., \cite[Theorem 7.3]{Brezis}),
$\Phi_t ^{0}$ is well-defined and is $C^1$ in $t$.
For $\Phi_t ^{0}$, we have the following comparison result.

\begin{prop}\label{prop:comparison}
Assume $u_1,u_2\in \Lambda_{0}^{\bd{p}}$. If $u_1\leq u_2$ and $u_1\neq u_2$, then $\Phi_{t}^{0}(u_1)< \Phi_{t}^{0}(u_2)$
for all $t>0$.
\end{prop}
The proof of Proposition \ref{prop:comparison} follows from \cite[Theorem 6.2]{mramor} with slight modifications.
For the reader's convenience, we provide the proof of Proposition \ref{prop:comparison} in Appendix \ref{sec:app1}.
Result similar to Proposition \ref{prop:comparison} also appears in \cite{dev2007}.
In \cite{dev2007}, one need a ``transitive'' condition (\cite[p. 2414, line 7]{dev2007}).
In our settings, (S\ref{eq:S3}) ensures this condition.

As in \cite{Rabi2007},
we choose a subset of $\Lambda_{0}^{\bd{p}}$ to prove the deformation lemma.
Set
\begin{equation*}
\MG^{\bd{p}}_{0} =\left\{u \in \Lambda_{0}^{\bd{p}} \, |\, 0\leq u \leq w_{0}-v_0 \right\}.
\end{equation*}
It is easy to see that $\MG^{\bd{p}}_{0}$ is a compact set with respect to the norm $\norm{\cdot}_{\Lambda_{0}^{\bd{p}}}$,
as shown in the following proposition.
\begin{prop}\label{prop:compact}
$\MG^{\bd{p}}_{0}$ is compact with respect to the norm $\norm{\cdot}_{\Lambda_{0}^{\bd{p}}}$.
\end{prop}
\begin{proof}
Assume $(u_k) \subset \MG^{\bd{p}}_{0}$.
By the definition of $\norm{\cdot}_{\Lambda_{0}^{\bd{p}}}$, we only need to prove that $(u_k)$ is compact with respect to pointwise convergence, which follows from Lemma \ref{lem:subsequ}.
\end{proof}
\begin{rmk}
Proposition \ref{prop:compact} will be used in the proof of deformation lemma (please see Lemma \ref{lem:2.2} \eqref{lem:2.2-5} below).
In the proof of \cite[Proposition 3.6]{Rabi2006}, Bolotin and Rabinowitz
obtained ``compactness'' by verifying
the corresponding functional
satisfied
Palais-Smale condition.
In our settings, the compactness condition is directly obtained.
\end{rmk}
We have the following deformation lemma.
\begin{lem}\label{lem:2.2}
For the semiflow $\Phi_{t}^{0}$ defined in \eqref{eq:dela7}, we have:
\begin{enumerate}[(i)]
  \item \label{lem:2.2-1} $\Phi_{t}^{0}(u)=u$ if $(I_{0}^{\bd{p}})'(u)=0$;
  \item \label{lem:2.2-2} $I_{0}^{\bd{p}}\left(\Phi_{t}^{0}(u)\right) \leq I_{0}^{\bd{p}}(u)$;
  \item \label{lem:2.2-3} $\Phi_{t}^{0} \MG^{\bd{p}}_{0} \subset \MG^{\bd{p}}_{0}$;
  \item \label{lem:2.2-4} For any $u \in \MG^{\bd{p}}_{0}$, there is a sequence $(t_i)\subset \R$ with $t_i\to \infty$ as $i\to \infty$ such that
  $\Phi_{t_i}^{0}(u)\to U$ pointwise for some $U\in \MG^{\bd{p}}_{0}$, and
  $I_{0}^{\bd{p}}(U)=\lim_{t\to\infty}I_{0}^{\bd{p}}(\Phi_{t}^{0}(u))$, and $U+v_0$ is a solution of \eqref{eq:PDE};
  \item \label{lem:2.2-5} If $\MK_{c}:=\{u \in \MG^{\bd{p}}_{0} \,|\, I_{0}^{\bd{p}}(u)=c, (I_{0}^{\bd{p}})'(u)=0\}=\emptyset$,
  there  is an $\epsilon>0$ such that $\Phi_{1}^{0}((I_{0}^{\bd{p}})^{c+\epsilon}) \subset (I_{0}^{\bd{p}})^{c-\epsilon}$,
  where $(I_{0}^{\bd{p}})^{t}:=\{u\in \MG^\bd{p}_{0}\,|\, I_{0}^{\bd{p}}(u)\leq t\}$.
\end{enumerate}
\end{lem}
In \cite[Lemma 8.6]{mramor}, the authors established Mountain Pass Theorem by imposing
a condition that the functional was a Morse function.
This condition is used to prove a similar property of Lemma \ref{lem:2.2} \eqref{lem:2.2-5}.
We prove Lemma \ref{lem:2.2} in Appendix \ref{sec:app1}.
Now set
\begin{equation*}
\MH^{\bd{p}}_{0} =\left\{h \in C\left([0,1], \MG^{\bd{p}}_{0}\right) \,|\, h(0)=0, h(1)=w_{0}-v_0\right\}
\end{equation*}
and
\begin{equation*}
d_{0}^{\bd{p}}=\inf _{h \in \MH^{\bd{p}}_{0}} \max _{\theta \in[0,1]} I_{0}^{\bd{p}}(h(\theta)).
\end{equation*}
\begin{prop}\label{thm:2.1-1-99}
$d_{0}^{\bd{p}}>c_{0}^{\bd{p}}$.
\end{prop}
\begin{proof}
By Lemma \ref{prop:2.2369852145},
for any $u \in \Lambda_{0}^{\bd{p}}$,
\begin{equation}\label{eq:2.2}
I_{0}^{\bd{p}}(u) \geq c_{0}^{\bd{p}}.
\end{equation}
So $d_{0}^{\bd{p}}\geq c_{0}^{\bd{p}}$.
Suppose, by contradiction, $d_{0}^{\bd{p}}=c_{0}^{\bd{p}}$.
Then there exist $h_j\in \MH^{\bd{p}}_{0}$ and $\sigma_j\in (0, 1)$
such that
\begin{equation}\label{eq:2.3}
\max _{\theta \in[0,1]} I_{0}^{\bd{p}}\left(h_{j}(\theta)\right) \rightarrow c_{0}^{\bd{p}}\quad \quad \textrm{as  }j\to\infty
\end{equation}
and
\begin{equation}\label{eq:2.4}
h_{j}\left(\sigma_{j}\right) (\bd{0})=\frac{1}{2} \left(w_{0}-v_{0}\right)(\bd{0}) .
\end{equation}
By \eqref{eq:2.2}-\eqref{eq:2.3}, we have
\begin{equation}\label{eq:2.5}
I_{0}^{\bd{p}}\left(h_{j}\left(\sigma_{j}\right)\right) \rightarrow c_{0}^{\bd{p}}\quad\quad \textrm{as  }j\to\infty.
\end{equation}
Since $h_j(\sigma_j)\in \MG^{\bd{p}}_{0}$, a compact set by Proposition \ref{prop:compact},
$h_j(\sigma_j)$ has a subsequence (still denoted by $h_j(\sigma_j)$) which
converges
in $\Lambda_{0}^{\bd{p}}$ to
$U\in \MG^{\bd{p}}_{0}$. Since $I_{0}^{\bd{p}}$ is continuous on $\Lambda_{0}^{\bd{p}}$, by \eqref{eq:2.2} and \eqref{eq:2.5},
\begin{equation*}
c_{0}^{\bd{p}} \leq I_{0}^{\bd{p}}(U) =\lim _{j \rightarrow \infty} I_{0}^{\bd{p}}\left(h_{j}\left(\sigma_{j}\right)\right)=c_{0}^{\bd{p}}.
\end{equation*}
Then
$U \in \mathcal{M}_{0}^{\bd{p}}= \mathcal{M}_{0}$.
Hence
$U+v_0=v_0$ or $U+v_0=w_{0}$.
But by \eqref{eq:2.4},
\begin{equation*}
 U(\bd{0})=\frac{1}{2} \left(w_{0}-v_{0}\right)(\bd{0}),
\end{equation*}
a contradiction. So $d_{0}^{\bd{p}}> c_{0}^{\bd{p}}$.
\end{proof}

\begin{thm}\label{thm:2.1}
$d_{0}^{\bd{p}}$ is a critical value of $I_{0}^{\bd{p}}$ on $\Lambda_{0}^{\bd{p}}$ with
a corresponding critical point $u_{\bd{p}}$ satisfying $0< u_{\bd{p}} < w_0-v_0 $ and $u_{\bd{p}}+v_0$ is a solution of \eqref{eq:PDE}.
\end{thm}

\begin{proof}
Suppose, by contradiction, $\MK_{d_{0}^{\bd{p}}}= \emptyset$.
Lemma \ref{lem:2.2} \eqref{lem:2.2-5} implies that there exists $\epsilon>0$ such that
\begin{equation}\label{eq:haha2}
  \Phi_{1}^{0}((I_{0}^{\bd{p}})^{d_{0}^{\bd{p}}+\epsilon}) \subset (I_{0}^{\bd{p}})^{d_{0}^{\bd{p}}-\epsilon}.
\end{equation}
By the definition of $d_{0}^{\bd{p}}$, there is an
$h \in \MH^{\bd{p}}_{0}$
satisfying
\begin{equation*}
\max _{\theta \in[0,1]} I_{0}^{\bd{p}}(h(\theta)) \leq d_{0}^{\bd{p}}+\epsilon.
\end{equation*}
Then by Lemma \ref{lem:2.2} \eqref{lem:2.2-1} and \eqref{lem:2.2-3}, we have
$\Phi_{1}^{0} \circ h \in \MH^{\bd{p}}_{0}$.
But by \eqref{eq:haha2},
$$
\max _{\theta \in[0,1]} I_{0}^{\bd{p}}\left(\Phi_{1}^{0} \circ h(\theta)\right) \leq d_{0}^{\bd{p}}-\epsilon,
$$
which is impossible by
the definition of $d_{0}^{\bd{p}}$.
Thus $\MK_{d_{0}^{\bd{p}}}\neq \emptyset$ and $d_{0}^{\bd{p}}$
is a critical value with a corresponding critical point $u_{\bd{p}}\in \MG^{\bd{p}}_{0}$,
so $0\leq u_{\bd{p}}\leq w_0 -v_0$.
Since $u_{\bd{p}}+v_0$ is a solution of \eqref{eq:PDE}, by Lemma \ref{lem:2.5miao}, $v_0< u_{\bd{p}}+v_0< w_0 $.
\end{proof}
\begin{rem}
Another proof of Theorem \ref{thm:2.1} by heat flow method is provided in Appendix \ref{sec:app1}.
We prefer the above argument because it is more intuitive.
\end{rem}

As in \cite{Rabi2007}, we prove that $d_{0}^{\bd{p}}$ is indeed a mountain pass critical value as follows.
Set
\begin{equation*}
  \bar{\MH}^{\bd{p}}_{0}=\{\bar{h}\in C([0,1], \Lambda_{0}^{\bd{p}})\,|\, \bar{h}(0)=0, \bar{h}(1)=w_0-v_0\}
\end{equation*}
and
\begin{equation*}
\bar{d}_{0}^{\bd{p}}=\inf _{\bar{h} \in \bar{\MH}^{\bd{p}}_{0}} \max _{\theta \in[0,1]} I_{0}^{\bd{p}}\left(\bar{h}(\theta)\right).
\end{equation*}
So $\bar{d}_{0}^{\bd{p}}$ is a classical mountain pass critical value. We have:
\begin{prop}
$d_{0}^{\bd{p}}=\bar{d}_{0}^{\bd{p}}$.
\end{prop}
\begin{proof}
Obviously $d_{0}^{\bd{p}}\geq \bar{d}_{0}^{\bd{p}}$.
To prove the converse inequality, for any
$\bar{h}\in \bar{\MH}^{\bd{p}}_{0}$, set
$$
h(\theta)=\max \{\min \left(\bar{h}(\theta), w_{0}-v_0\right), 0\}.
$$
Then $h \in \MH^{\bd{p}}_{0}$ and
\begin{equation*}
  \begin{split}
    &I_{0}^{\bd{p}}(h(\theta)) +c_{0}^{\bd{p}}\\
     \leq &I_{0}^{\bd{p}}(h(\theta)) +  I_{0}^{\bd{p}}(\min \left(\min \left(\bar{h}(\theta), w_{0}- v_{0}\right),0\right))\\
       \leq& I_{0}^{\bd{p}}(\min \left(\bar{h}(\theta), w_{0}-v_{0}\right))+ I_{0}^{\bd{p}}(0)\\
      =&I_{0}^{\bd{p}}(\min \left(\bar{h}(\theta), w_{0}-v_{0}\right))+c_{0}^{\bd{p}},
  \end{split}
\end{equation*}
where the first inequality follows from Lemma \ref{prop:2.2369852145} and the second follows from Lemma \ref{lem:2.6miao}.
Similarly, we have
\begin{equation*}
  \begin{split}
    &I_{0}^{\bd{p}}(\min \left(\bar{h}(\theta), w_{0}-v_0\right)) +c_{0}^{\bd{p}}\\
     \leq &I_{0}^{\bd{p}}(\min \left(\bar{h}(\theta), w_{0}-v_0\right))+ I_{0}^{\bd{p}}(\max \left(\bar{h}(\theta), w_{0}-v_0\right))\\
    \leq & I_{0}^{\bd{p}}(\bar{h}(\theta))+I_{0}^{\bd{p}}(w_{0}-v_0)\\
       =&I_{0}^{\bd{p}}(\bar{h}(\theta))+c_{0}^{\bd{p}},
  \end{split}
\end{equation*}
so
$I_{0}^{\bd{p}}(h(\theta)) \leq I_{0}^{\bd{p}}\left(\bar{h}(\theta)\right)$
for each $\theta$.
\end{proof}
\begin{rem}
As is well-known, mountain pass solutions have Morse index $1$.
Thus part of the arguments in \cite{mramor} can be simplified.
For instance, the proofs of \cite[Theorem 8.4, Lemma 8.6]{mramor}.
\end{rem}

There is another candidate of solution, the maximum of $I_{0}^{\bd{p}}$ on $\MG_{0}^{\bd{p}}$, in the gap of $\MM_0$.
\begin{prop}\label{prop:ldkd}
$I_{0}^{\bd{p}}$ attains the maximum, say $\hat{u}_{\bd{p}}$ on $\MG_{0}^{\bd{p}}$.
If $\bd{j}\in \T^{\bd{p}}_{0}$ such that
$0<\hat{u}_{\bd{p}}(\bd{j})<(w_0-v_0)(\bd{j})$, then
\begin{equation}\label{eq:l111kfg22flkg}
\sum_{\bd{k}:\norm{\bd{k}-\bd{j}}\leq r}\partial_{\bd{k}}S_{\bd{j}}(\hat{u}_{\bd{p}}+v_0)=0,
\end{equation}
i.e.,
$\hat{u}_{\bd{p}}$ satisfies \eqref{eq:PDE} at $\bd{j}$.
\end{prop}
\begin{proof}
By the definition of $\MG_{0}^{\bd{p}}$ and
Proposition \ref{prop:compact},
there is $\hat{u}_{\bd{p}}\in \MG_{0}^{\bd{p}}$ such that
$$I_{0}^{\bd{p}}(\hat{u}_{\bd{p}})=\sup_{u\in \MG_{0}^{\bd{p}}}I_{0}^{\bd{p}} (u).$$
If $\bd{j}\in \T^{\bd{p}}_{0}$ such that
$0<\hat{u}_{\bd{p}}(\bd{j})<(w_0-v_0)(\bd{j})$, then
choose $t_0>0$ such that  $0<\hat{u}_{\bd{p}}(\bd{j})+t<(w_0-v_0)(\bd{j})$ hold for all $|t|<t_0$.
For $\bd{k}\in \T^{\bd{p}}_{0}$,
define
\begin{equation*}
  v(\bd{k})=\left\{
              \begin{array}{ll}
                t, &  \bd{k}=\bd{j},\\
                0, & \bd{k}\neq\bd{j},
              \end{array}
            \right.
\end{equation*}
and extend $v$ to be a period function in $\Lambda_{0}^{\bd{p}}$.
Then by \eqref{eq:derivative},
\begin{equation}\label{eq:lkfgflkg}
  0=(I_{0}^{\bd{p}})'(\hat{u}_{\bd{p}})v=t \sum_{\bd{k}:\norm{\bd{k}-\bd{j}}\leq r}\partial_{\bd{k}}S_{\bd{j}}(\hat{u}_{\bd{p}}+v_0),
\end{equation}
where the first equality follows from $\hat{u}_{\bd{p}}$ is a maximum point of $I_{0}^{\bd{p}}$ on $\MG_{0}^{\bd{p}}$.
Since \eqref{eq:lkfgflkg} holds for any $|t|<t_0$,
we have \eqref{eq:l111kfg22flkg}.
\end{proof}

\begin{rem}
If $\bd{j}\in \T^{\bd{p}}_{0}$ such that
$0=\hat{u}_{\bd{p}}(\bd{j})$, then
choose $t_0>0$ such that  $0<\hat{u}_{\bd{p}}(\bd{j})+t<(w_0-v_0)(\bd{j})$ hold for all $0<t<t_0$.
The argument in Proposition \ref{prop:ldkd}
shows that $\sum_{\bd{k}:\norm{\bd{k}-\bd{j}}\leq r}\partial_{\bd{k}}S_{\bd{j}}(\hat{u}_{\bd{p}}+v_0)\leq 0$.
\end{rem}

If $\hat{u}_{\bd{p}}$ obtained in Proposition \ref{prop:ldkd}
is a solution of \eqref{eq:PDE}, then by Lemma \ref{lem:2.5miao},
$0<\hat{u}_{\bd{p}}<w_0-v_0$.
Conversely, if $0<\hat{u}_{\bd{p}}<w_0-v_0$, by Proposition \ref{prop:ldkd}, $\hat{u}_{\bd{p}}$ is a solution of \eqref{eq:PDE}.
Unfortunately, we do not know whether $\hat{u}_{\bd{p}}$ obtained in Proposition \ref{prop:ldkd} is a solution of \eqref{eq:PDE} for general $\bd{p}$.

But for $\bd{p}=(1,\cdots,1)=:\bd{1}$, we see $0<\hat{u}_{\bd{1}}<w_0-v_0$, thus $\hat{u}_{\bd{1}}$ is a solution of \eqref{eq:PDE}.
In fact, as one can easily see, $\hat{u}_{\bd{p}}$ is same to the mountain pass solution of Theorem \ref{thm:2.1}.

For $\bd{p}\neq (1,\cdots,1)$, $\hat{u}_{\bd{p}}$ may be coincide with $\hat{u}_{\bd{1}}$ and it may not give more solutions.
To see this, we examine the classical Frenkel-Kontorova model.
\begin{exam}
Let $n=2$ and $\bd{p}(1)=(1,0)$, $\bd{p}(2)=(1,1)$.
Set
\begin{equation*}
  S_{\bd{0}}(u):=s(u|_{B_{\bd{0}}^{1}}):=\sin(2\pi u(\bd{0}))+\frac{1}{16}\sum_{\|\bd{j}\|=1}[u(\bd{j})-u(\bd{0})]^2.
\end{equation*}
Then $\inf_{u\in \Lambda_{0}^{\bd{p}(1)}}S_{\bd{0}}(u)=-1$ is attained at $k-\frac{1}{4}$ ($k\in\Z$)
and $\sup_{u\in \Lambda_{0}^{\bd{p}(1)}}S_{\bd{0}}(u)=1$
is attained at $k+\frac{1}{4}$ ($k\in\Z$).
Assume the gap pair is $v_0=-\frac{1}{4}$ and $w_0=\frac{3}{4}$.
So $I_{0}^{\bd{p}(1)}(u):=S_{\bd{0}}(u-\frac{1}{4})$ and
$\inf_{u\in \Lambda_{0}^{\bd{p}(1)}}I_{0}^{\bd{p}(1)}(u)=-1$ is attained at $k$ ($k\in\Z$)
and $\sup_{u\in \Lambda_{0}^{\bd{p}(1)}}I_{0}^{\bd{p}(1)}(u)=1$
is attained at $k+\frac{1}{2}$ ($k\in\Z$).

On the other hand,
\begin{equation*}
\begin{split}
I_{0}^{\bd{p}(2)}(u):=&S_{\bd{0}}(u-\frac{1}{4})+S_{\bd{0}}(\tau^{1}_{-1}u-\frac{1}{4})\\
=&\frac{1}{4}[u((1,0))-u((0,0))]^2-\cos(2\pi u((0,0)))-\cos(2\pi u((1,0))).
\end{split}
\end{equation*}
In the gap of $v_0,w_0$, there is one locally maximum point $u\equiv \frac{1}{2}$, and there are two mountain pass critical points $\bar{u}$ and $\tau^{1}_{-1}\bar{u}$
but no other locally maximum point exists.
Please see Figure \ref{fig:1}
for the graph of $I_{0}^{\bd{p}(2)}$.
\end{exam}
\begin{figure}
    \begin{minipage}[t]{0.8\linewidth}
    \centering
    \includegraphics[width=100mm,height=87mm]{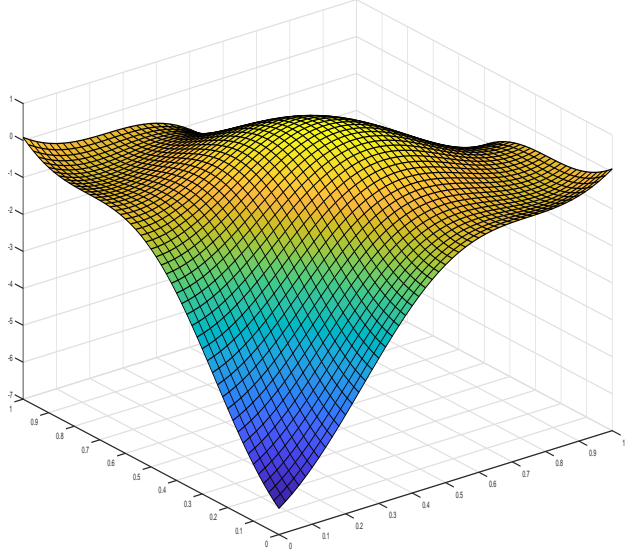}
    \caption{Graph of the function $I_{0}^{\bd{p}(2)}$}\label{fig:1}
    \end{minipage}
\end{figure}

\bigskip
\subsection{Multiplicity of periodic mountain pass solutions}\label{sec:multimountain}
\

We shall prove that
varying $\bd{p}$ will produce more periodic mountain pass solutions.
This is different with periodic minimal and Birkhoff solutions (cf. Section \ref{sec:4524573}, Lemma \ref{prop:2.2369852145}).
Toward this end,
for
$k\in\N$, assume $\bd{p}(k)=(k,1,\cdots,1)\in\N^{n}$.
By Theorem \ref{thm:2.1}, there exists
$u_{k}^*\in \MG_{0}^{\bd{p}(k)}$ such that $u_k ^* +v_0$
is a mountain pass solution.
Since $\MG_{0}^{\bd{p}(1)}\subset \MG_{0}^{\bd{p}(k)}$, $u_1 ^*\in \MG_{0}^{\bd{p}(k)}$.
It is possible that $u_k ^*=u_1 ^*$.
But we shall show this cannot happen for infinite many $k$.
First we have
\begin{prop}\label{prop:2.3}
There is a constant $M_0$, independent of $k$, such that $$0<d_{0}^{\bd{p}(k)}-c_{0}^{\bd{p}(k)}\leq M_0.$$
\end{prop}
\begin{proof}
To estimate $d_{0}^{\bd{p}(k)}$ we need to construct a suitable $h_k\in \MH_{0}^{\bd{p}(k)}$.
To do so, we first define the following $\chi_k$ for $k\geq 2$.

If
$k$ is even,
set
$$
\chi_{k}(t, 0)=
\left\{
\begin{array}{ll}
{t,} & {0\leq t\leq 1}, \\
 {1,} & {1  \leq t \leq \frac{k+5}{2}}.
 \end{array}
 \right.
$$
For $1 \leq i \leq \frac{k}{2}-1$, set
$$
\chi_{k}(t, i)=\left\{
\begin{array}{ll}
{0,} & {0 \leq t \leq i+\frac{1}{2}}, \\
{2 t-1-2 i,} & {i+\frac{1}{2} \leq t \leq i+1}, \\
{1,} & {i+1 \leq t \leq \frac{k+5}{2}}.
\end{array}\right.
$$
Lastly
set
$$
\chi_{k}(t, \frac{k}{2})=\left\{\begin{array}{ll}
{0,} & {0 \leq t \leq \frac{k+1}{2}} ,\\
{1-\frac{1}{2}(\frac{k+5}{2}-t),} & {\frac{k+1}{2} \leq t\leq \frac{k+5}{2}},
\end{array}\right.
$$
and $\chi_k(t,i)=\chi_k(t,-i)$ for $-\frac{k}{2} < i<0$.
Extend $\chi_k(t, \cdot)$ as a $k$-periodic function on $\Z$.

If $k$ is odd, then $k-1$ is even and thus $\chi_{k-1}(t,i)$ is well-defined for $-\frac{k-1}{2}< i\leq \frac{k-1}{2}$ and $0\leq t \leq \frac{k+5}{2}$ by the previous paragraph.
Now define $\chi_{k}(t,i):=\chi_{k-1}(t,i)$ for $0\leq i\leq \frac{k-1}{2}$ and $0\leq t \leq \frac{k+5}{2}$.
Let $\chi_k(t,i):=\chi_k(t,-i)$ for $-\frac{k-1}{2}\leq i<0$ and
then
extend $\chi_k(t, \cdot)$ as a $k$-periodic function on $\Z$.

For the above $\chi_k$,
set
\begin{equation}\label{eq:phik}
  \phi_{k}(\theta, i)=\chi_{k}\left(\frac{\theta(k+5)}{2}, i\right).
\end{equation}
Letting
$
h_{k}(\theta)(\bd{i})=\phi_{k}(\theta, \bd{i}_{1})[w_{0}(\bd{i})-v_{0}(\bd{i})]
$
gives $h_{k} \in H_{0}^{\bd{p}(k)}$.
Notice that for $\theta\in [0,1]$, $h_{k}(\theta) \neq 0$, $w_{0}-v_{0}$
on $\{\bd{i}\in\Z^n\,|\, 0\leq \bd{i}_1\leq k,\,\bd{i}_2=\bd{i}_3=\cdots=\bd{i}_n=0\}$ for at most two points,
and for any $\bd{i}$, $h_k(\theta)(\bd{i})$ is monotone nondecreasing in $\theta$.
Thus we have
\begin{equation*}
\max _{\theta \in[0,1]} I^{\bd{p}(k)}_{0}\left(h_{k}(\theta)\right) -c_{0}^{\bd{p}(k)}\leq M_0
\end{equation*}
for some $M_0$ independent of $k$.
Hence for $k\geq 2$, $d_{0}^{\bd{p}(k)} -c_{0}^{\bd{p}(k)}\leq M_0$.
Enlarging $M_0$ if necessary shows that $d_{0}^{\bd{p}(k)} -c_{0}^{\bd{p}(k)}\leq M_0$ hold for all $k\in\N$,
thus Proposition \ref{prop:2.3} is proved.
\end{proof}

Now we prove the multiplicity of periodic mountain pass solutions.
\begin{thm}\label{thm:2.4}
The set $\{u_{k}^{*}\}_{k=1}^{\infty}$ is infinite.
\end{thm}
\begin{proof}
Suppose, by contradiction, $\{u_{k}^{*}\}_{k=1}^{\infty}=\{u_{k_i}^{*}\}_{i=1}^{m}$ and
$1=k_{1}<\cdots<k_{m}$. For $l \in \mathbb{N}$, %
we obtain 
$u_{l k_{m}}^{*}=u_{k_{j}}^{*}$ for some $j\in\{1, \dots, m\}$.
Hence
\begin{equation*}
\begin{split}
d_{0}^{\bd{p}(l k_{m})}=&I^{\bd{p}(l k_{m})}_{0}(u_{l k_{m}}^{*})=I^{\bd{p}(l k_{m})}_{0}(u_{k_{j}}^{*})\\
=&\gamma_{k_{m}} d_{0}^{\bd{p}(k_{j})} \geq \gamma_{k_{m}} \min _{1 \leq i \leq m} d_{0}^{\bd{p}(k_{i})}\\
\geq & l \min _{1 \leq i \leq m} d_{0}^{\bd{p}(k_{i})}.
\end{split}
\end{equation*}
Proposition \ref{thm:2.1-1-99} implies $d_{0}^{\bd{p}(k_{i})}>c_{0}^{\bd{p}(k_{i})}$, $1 \leq i \leq m$
and then $$\min _{1 \leq i \leq m} d_{0}^{\bd{p}(k_{i})}- \min_{1 \leq i \leq m}c_{0}^{\bd{p}(k_{i})}>0.$$
Thus
\begin{equation*}
\begin{split}
d_{0}^{\bd{p}(l k_{m})}-c_{0}^{\bd{p}(l k_{m})} \geq &\gamma_{k_{m}} \min _{1 \leq i \leq m} d_{0}^{\bd{p}(k_{i})} -  lc_{0}^{\bd{p}(k_m)} \\
\geq&\gamma_{k_{m}} \min _{1 \leq i \leq m} d_{0}^{\bd{p}(k_{i})} -  l\min_{1 \leq i \leq m}c_{0}^{\bd{p}(k_{i})}\\
\geq &l\left( \min _{1 \leq i \leq m} d_{0}^{\bd{p}(k_{i})}- \min_{1 \leq i \leq m}c_{0}^{\bd{p}(k_{i})}\right)\\
\to& \infty
\end{split}
\end{equation*}
as $l\to \infty$,
contrary to Proposition \ref{prop:2.3}.
\end{proof}

If $\bd{p}(k)=(k,1, \cdots, 1)$ is replaced by $(1,\cdots, k,\cdots, 1)$,
similar arguments of this section give more periodic mountain pass solutions.
What happens if we change the coordinate systems $(\bd{e}_1,\cdots, \bd{e}_n)$?
we learn from \cite[Lemma 5.4]{LC} that if $\omega_i=\sum_{j=1}^{n}\alpha_{ij}\bd{e}_j$ with $\alpha_{ij}\in\Z$ and
the vectors $\omega_i$ are linearly independent, then we do not obtain more periodic minimal and Birkhoff solutions.
Different with minimal and Birkhoff solutions, changing coordinate systems produces more periodic mountain pass solutions.
To see this,
for simplicity set $\omega_1=\bd{e}_1+\bd{e}_2$, $\omega_2=-\bd{e}_1+\bd{e}_2$, and $\omega_j=\bd{e}_j$ for $j=3,\cdots,n$.
Set $\bd{p}(k,i)=(1,\cdots,1)+(k-1)\bd{e}_i$.
Denote by $\MS$ the set of critical points of $I_{0}^{\bd{p}(k,i)}$ on $\Lambda_{0}^{\bd{p}(k,i)}$ in the
coordinate sysetems $(\bd{e}_1,\cdots,\bd{e}_n)$.
Let $u_{k}^{*}$ be the critical point that are $k$-periodic in $\omega_1$
and $1$-periodic in $\omega_2, \cdots, \omega_n$. Assume $u_{k}^{*}\not\in \R^{\Z^n/\Z^n}$.
Notice that by Proposition \ref{prop:2.3} and Theorem \ref{thm:2.4} there are infinitely many functions of this type.
We have:
\begin{prop}
$u^* _k\not\in\MS$.
\end{prop}
\begin{proof}
Suppose, by contradiction, %
$u^* _k=u$ for some $u\in\MS$.
Since $u\in\MS$, there is some $j\in\N$ such that
\begin{equation}\label{eq:240930}
\left\{
  \begin{array}{ll}
    u\left(\bd{i}+j \bd{e}_{p}\right)=u(\bd{i}), & \,  \\
    u(\bd{i}+ \bd{e}_{i})=u(\bd{i}), & \hbox{for all $i\neq p$.}
  \end{array}
\right.
\end{equation}

\begin{itemize}
  \item If
$p\neq 1,2$, then
\begin{equation}\label{eq:222200}
u(\bd{i}+\bd{e}_{p})=u_{k}^{*}(\bd{i}+\bd{e}_{p})=u_{k}^{*}(\bd{i}+\omega_{p})=u_{k}^{*}(\bd{i})=u(\bd{i}),
\end{equation}
Combining \eqref{eq:240930} and \eqref{eq:222200} gives $u\in\R^{\Z^n/\Z^n}$.
But this contradicts $u=u_{k}^{*}\not\in\R^{\Z^n/\Z^n}$.
\item If
$p = 2$, then
\begin{equation*}
\begin{split}
u(\bd{i}+\bd{e}_{2})=&u(\bd{i}+\bd{e}_{2}-\bd{e}_{1})=u_{k}^{*}(\bd{i}+\bd{e}_{2}-\bd{e}_{1})\\
=&u_{k}^{*}(\bd{i}+\bd{\xi}_2)=u_{k}^{*}(\bd{i})=u(\bd{i}).
\end{split}
\end{equation*}
Again $u\in\R^{\Z^n/\Z^n}$, a contradiction.
\item If $p = 1$, then
\begin{equation*}
\begin{split}
u(\bd{i}-\bd{e}_1)=&u(\bd{i}-\bd{e}_1+\bd{e}_2)=u(\bd{i}+\bd{\xi}_2)\\
=&u^{*}_{k}(\bd{i}+\bd{\xi}_2)=u^{*}_{k}(\bd{i})=u(\bd{i}),
\end{split}
\end{equation*}
again a contradiction.
\end{itemize}
\end{proof}

Certain classes of sets which consist of solutions of \eqref{eq:PDE} attracts researchers' attention.
When the elements in the set have good order property, the set becomes a foliation or lamination.
With the periodic mountain pass solutions in hand, one may wonder: is there a possibility that the periodic mountain pass solutions constitute a foliation or lamination?
Unfortunately, the answer is ``negative''.
(In fact, 
we construct some periodic solutions that are cross.
These periodic solutions are suspected to be mountain pass type.
Please see Remark \ref{rem:3.33000} below.)
To this end,
we need the following definition.
\begin{defi}\label{def:touch}
We say $u$ touches $v$ from below (resp. above) if $u\leq v$ (resp. $v\leq u$), $u\neq v$ and there exists $\bd{i}$ such that $u(\bd{i})=v(\bd{i})$.
We say $u$ intersects $v$ if there are $\bd{i}$, $\bd{j}\in\Z^n$ such that $[u(\bd{i})-v(\bd{i})][u(\bd{j})-v(\bd{j})]<0$.
\end{defi}
\begin{rem}
Note that if $u\leq v$ are solutions of \eqref{eq:PDE}, by Lemma \ref{lem:2.5miao} $u$ will not touch $v$ from below.
\end{rem}
Set
\begin{equation*}
 \begin{split}
   \hat{\MH}_{0}^{\bd{p}}=  \{h \in \MH^{\bd{p}}_{0} \,|\, &h \text { is monotone nondecreasing in } \theta, i.e.,  \\
     & h(\theta_1)\geq h(\theta_2) \textrm{ for  any  }\theta_1>\theta_2\}.
 \end{split}
\end{equation*}
and
$$
\hat{d}_{0}^{\bd{p}}=\inf _{h \in \hat{\MH}_{0}^{\bd{p}}} \max _{\theta \in[0,1]} I_{0}^{\bd{p}}(h(\theta)).
$$
Then $\hat{d}_{0}^{\bd{p}} \geq d_{0}^{\bd{p}}$.
Similar to Theorem \ref{thm:2.1}, $\hat{d}_{0}^{\bd{p}}$ is a critical
value of $I_{0}^{\bd{p}}$ with a corresponding critical point in $\MG^{\bd{p}}_{0}$. %
By Proposition \ref{prop:comparison} and Lemma \ref{lem:2.2}, $\Phi_{t}^{0}(\hat{\MH}_{0}^{\bd{p}}) \subset \hat{\MH}_{0}^{\bd{p}}$, where $\Phi_{t}^{0}: \Lambda^{\bd{p}(k)}_{0}\to \Lambda^{\bd{p}(k)}_{0}$ is defined as in \eqref{eq:dela7}.
For $\bd{p}(k)=(k,1,\cdots,1)\in\N^{n}$,
set $h\in \hat{\MH}_{0}^{\bd{p}(k)}$. For $t>0$, define
\begin{equation*}
\begin{split}
&\underline{\theta}_{t}:=\sup \{\theta \in[0,1] \,|\, \Phi_{t}^{0} h(\theta) \leq u_{0}, \textrm{ but } \Phi_{t}^{0} h(\theta) \neq u_{0}\} \\
\leq& \overline{\theta}_{t}:=\inf \{\theta \in[0,1] \,|\, \Phi_{t}^{0} h(\theta) \geq u_{0}, \textrm{ but } \Phi_{t}^{0} h(\theta) \neq u_{0}\},
\end{split}
\end{equation*}
where 
$u_0\in \Lambda_{0}^{\bd{p}(1)}$ satisfying $0<u_0<w_0-v_0$ and $u_0+v_0$ is a solution of \eqref{eq:PDE}. %
By the periodicity of $u_0$ and $\Phi_{t}^{0} h(\theta)$,
$\Phi_{t}^{0} h(\theta)\leq u_0$ means
\begin{enumerate}[(a)]
\item \label{eq:0lf} $\Phi_{t}^{0} h(\theta)<u_0$, or
\item $\Phi_{t}^{0} h(\theta)$ touches $u_0$ from below, or
\item $\Phi_{t}^{0} h(\theta)=u_0$.
\end{enumerate}
For $\theta=\underline{\theta}_{t}$, \eqref{eq:0lf} will not hold.
We need the following lemma.
\begin{lem}\label{lem:003e}
For $\theta<\underline{\theta}_{t}$,
either $\Phi_t ^0 h(\theta)<u_0$ or $\Phi_t ^0 h(\theta)$ touches $u_0$ from below.
\end{lem}
\begin{proof}
%
For $\theta<\underline{\theta}_{t}$,
$h(\theta)\leq h(\underline{\theta}_{t})$ and then
$\Phi_{t}^{0} h(\theta)\leq \Phi_{t}^{0} h(\underline{\theta}_{t})\leq u_{0}$. We have the following three cases.
\begin{enumerate}
  \item Assume $h(\theta)< h(\underline{\theta}_{t})$. Then $\Phi_t ^0 h(\theta)< \Phi_t ^0 h(\underline{\theta}_{t})\leq u_0$ and we are throuth.
  \item Assume $h(\theta)$ touches $h(\underline{\theta}_{t})$ from below. Then $\Phi_t ^0 h(\theta)< \Phi_t ^0 h(\underline{\theta}_{t})\leq u_0$ and we are throuth.
  \item Assume $h(\theta)= h(\underline{\theta}_{t})$. Then $\Phi_t ^0 h(\theta)= \Phi_t ^0 h(\underline{\theta}_{t})\leq u_0$.
  \begin{enumerate}
    \item If $\Phi_t ^0 h(\underline{\theta}_{t})=u_0$, then so is $\Phi_t ^0 h(\theta)$, contradicting the definition of $\underline{\theta}_{t}$.
    \item If $\Phi_t ^0 h(\underline{\theta}_{t})$ touches $u_0$ from below, then so is $\Phi_t ^0 h(\theta)$.
  \end{enumerate}
\end{enumerate}
Thus either $\Phi_t ^0 h(\theta)<u_0$ or $\Phi_t ^0 h(\theta)$ touches $u_0$ from below.
\end{proof}
Since $h$ is continous, $\underline{\theta}_{t}$ is attained.
Similar results hold for $\overline{\theta}_{t}$.
Moreover,
$\underline{\theta}_{t}, \overline{\theta}_{t}$ have the following monotone property.
\begin{prop}\label{prop:3.2}
Assume $h\in \hat{\MH}_{0}^{\bd{p}(k)}$.
\begin{enumerate}
  \item \label{prop:3.2-1} If $h\neq u_0$ for all $\theta\in [0,1]$, then the map $t \mapsto \underline{\theta}_{t}$ (resp. $t \mapsto \overline{\theta}_{t}$) is monotone nondecreasing (monotone non-increasing).
  \item \label{prop:3.2-2} Assume there is some $\theta_0\in (0,1)$ such that $h(\theta)=u_0$.
   If such $\theta_0$ is unique, then $\underline{\theta}_t =\theta_0=\overline{\theta}_t$ for all $t>0$.
   If $h(\theta)=u_0$ for $\theta\in [a,b]\subset(0,1)$ and $[a,b]$ is the maximal interval owning this property, then $\underline{\theta}_t=a$ and $\overline{\theta}_t=b$.
\end{enumerate}
\end{prop}
\begin{proof}
By the definition of $\underline{\theta}_t, \overline{\theta}_t$, Proposition \ref{prop:3.2} \eqref{prop:3.2-2} is obvious.
Thus we only prove Proposition \ref{prop:3.2} \eqref{prop:3.2-1}.
Assume $t_1< t_2$. For $\mu\in(0, \underline{\theta}_{t_1})$,
by Lemma \ref{lem:2.2} \eqref{lem:2.2-1}, Lemma \ref{lem:003e} and Proposition \ref{prop:comparison},
$\{u\in \MG^{\bd{p}(k)}_{0}\,|\, 0\leq u\leq u_0\}$ is invariant under $\Phi_{t}^{0}$ and
$$
\Phi_{t+t_{1}}^{0} h(\mu)<u_{0}, \quad \textrm{for }t>0.
$$
If we take
$t=t_2 -t_1$, then
$\mu< \underline{\theta}_{t_2}$.
Hence $\underline{\theta}_{t_1}\leq  \underline{\theta}_{t_2}$.
%
%
$\overline{\theta}_{t_1}\geq \overline{\theta}_{t_2}$ can be proved similarly.
\end{proof}
\begin{rem}
We suspect that in the case \eqref{prop:3.2-1} of Proposition \ref{prop:3.2},
 either $\underline{\theta}_{t}$ is strictly increasing or $\underline{\theta}_{t}\equiv \theta_0$ for some $\theta_0\in (0,1)$.
To support this, suppose $t_1<t_2$.
\begin{itemize}
  \item If $\Phi_{t_{1}}^{0} h(\underline{\theta}_{t_1})$ touches $u_{0}$ from below, then $\underline{\theta}_{t_1}<  \underline{\theta}_{t_2}$.
  Indeed, by Proposition \ref{prop:comparison},
$\Phi_{t+t_{1}}^{0} h(\underline{\theta}_{t_1})< u_{0}$
for $t>0$.
In particular, $\Phi_{t_{2}}^{0}(h(\underline{\theta}_{t_{1}}))<u_{0}$.
If $\underline{\theta}_{t_1}= \underline{\theta}_{t_2}$
then $u_{0}>\Phi_{t_{2}}^{0}(h(\underline{\theta}_{t_{2}}))$,
a contradiction.
  \item If $\Phi_{t_{1}}^{0} h(\underline{\theta}_{t_1})=u_{0}$, then $\underline{\theta}_{t_1}=\underline{\theta}_{t_2}$.
  Indeed, suppose, by contradiction $\underline{\theta}_{t_1}<\underline{\theta}_{t_2}$.
  Then by Lemma \ref{lem:003e}, either $\Phi_{t_2}^{0}h(\underline{\theta}_{t_1})<u_0$ or
  $\Phi_{t_2}^{0}h(\underline{\theta}_{t_1})$ touches $u_0$ from below.
  Both cases contradict $\Phi_{t_{1}}^{0} h(\underline{\theta}_{t_1})=u_{0}$.
\end{itemize}
But we do not know if there are $t_1<t_2$ such that $\Phi_{t_{1}}^{0} h(\underline{\theta}_{t_1})$ touches $u_{0}$ from below and $\Phi_{t_{2}}^{0} h(\underline{\theta}_{t_2})=u_{0}$.
\end{rem}

\begin{thm}\label{thm:3.1}
Let $u_0 \in \Lambda_{0}^{\bd{p}(1)}$ such that $u_0+v_0$ is a solution of \eqref{eq:PDE} and $0< u_0 < w_0-v_0 $. Then for
$k\in\N$ sufficiently large, there exists a $u_k\in \Lambda^{\bd{p}(k)}_{0}$ with $u_k+v_0$ is a solution of
\eqref{eq:PDE} and $u_k$ intersects $u_0$.
\end{thm}

\begin{proof}
By Proposition \ref{prop:2.3} and its proof, for $k\in\N$ there are
$h_{\bd{p}(k)}\in\hat{\MH}_{0}^{\bd{p}(k)}$ and $M_0>0$ such that
$$\max _{0 \leq \theta \leq 1} I^{\bd{p}(k)}_{0}\left(h_{\bd{p}(k)}(\theta)\right)-kc_0\leq M_0.$$
By Proposition \ref{thm:2.1-1-99},
$d_{0}^{\bd{p}(1)}>c_{0}^{\bd{p}(1)}=c_0$.
Then for large $k$,
\begin{equation}\label{eq:3.2}
M_0<k(d_{0}^{\bd{p}(1)}-c_{0}^{\bd{p}(1)})=I^{\bd{p}(k)}_{0}\left(u_{0}\right)-kc_0.
\end{equation}
Since $\underline{\theta}_t, \overline{\theta}_t$ are monotone, we have
$$
\lim _{t \rightarrow \infty} \underline{\theta}_{t}=:\underline{\theta} \leq \overline{\theta}:=\lim _{t \rightarrow \infty} \overline{\theta}_{t}.
$$
Lemma \ref{lem:2.2} \eqref{lem:2.2-4} shows that there are a $u^{-}\in\Lambda_{0}^{\bd{p}(k)}$ and a sequence
$t_m \to\infty$ as $m\to\infty$ such that
\begin{equation*}
  u^{-}=\lim _{m \rightarrow \infty} \Phi_{t_{m}}^{0}(h_{\bd{p}(k)}(\underline{\theta})) %
\end{equation*}
and $u^- +v_0$ is a solution of \eqref{eq:PDE}.
Note that
by \eqref{eq:3.2} and Lemma \ref{lem:2.2} \eqref{lem:2.2-2}, 
\begin{equation}\label{eq:***}
\begin{split}
I^{\bd{p}(k)}_{0}(u^-)-kc_0
=&\lim_{m\to \infty}I^{\bd{p}(k)}_{0}(\Phi_{t_m}^{0}(h_{\bd{p}(k)}(\underline{\theta}))) -kc_0\\
\leq &I^{\bd{p}(k)}_{0}(h_{\bd{p}(k)}(\underline{\theta}))-kc_0\\
\leq &M_{0}\\
<& I^{\bd{p}(k)}_{0}(u_0) -kc_0,
\end{split}
\end{equation}
thus $I^{\bd{p}(k)}_{0}(u^-)< I^{\bd{p}(k)}_{0}(u_0)$ and then $u^{-}\neq u_0$.
We claim that $u^-$ intersects $u_0$.
Noticing that $u^- +v_0, u_0 +v_0$ are solutions of \eqref{eq:PDE}, if $u^-$ does not intersect $u_0$,
then by Lemma \ref{lem:2.5miao} either (i)
$u^->u_0$ or (ii) $u^-<u_0$.
If (i) holds,
by the definition of $\overline{\theta}_{t_m}$,
there is some $\bd{i}\in \Z^n$, such that $\Phi_{t_m}^{0}(h_{\bd{p}(k)}(\overline{\theta}))(\bd{i})\leq \Phi_{t_m}^{0}(h_{\bd{p}(k)}(\overline{\theta}_{t_m}))(\bd{i})= u_0(\bd{i})$.
But
by the periodicity of $\Phi_{t_{m}}^{0}(h_{\bd{p}(k)}(\underline{\theta}))$ and $u_0$, for $m$ large enough,
\begin{equation*} 
  \Phi_{t_{m}}^{0}(h_{\bd{p}(k)}(\overline{\theta})) \geq \Phi_{t_{m}}^{0}(h_{\bd{p}(k)}(\underline{\theta}))>u_{0},
\end{equation*}
a contradiction.
If (ii) is satisfied,
by the definition of $\underline{\theta}_{t_m}$, there exists $\bd{i}\in\Z^n$ such that
$\Phi_{t_{m}}^{0}(h_{\bd{p}(k)}(\underline{\theta}_{t_{m}}))(\bd{i})=u_{0}(\bd{i})$.
But
$$\Phi_{t_{m}}^{0}\left(h_{\bd{p}(k)}\left(\underline{\theta}_{t_{m}}\right)\right)\leq  \Phi_{t_{m}}^{0}\left(h_{\bd{p}(k)}\left(\underline{\theta}\right)\right)<u_{0}$$
for large $m$,
again a contradiction.
Thus Theorem \ref{thm:3.1} is proved by setting $u_k=u^{-}$.
\end{proof}

\begin{rmk}\label{rem:3.33000}
In the proof of Theorem \ref{thm:3.1},
one can obtain another solution, say $u^{+}+v_0$ by taking limit for a subsequence of $\Phi_{t}^{0}(h_{\bd{p}(k)}(\overline{\theta}))+v_0$.
Of course, it may happen that $u^{+}=u^{-}$.
We cannot figure out whether $u^-$ (or $u^+$) is a mountain pass critical point from the proof of Theorem \ref{thm:3.1},
but please see \cite[Remark 3.3]{Rabi2007} for more discussions.
\end{rmk}

Since $\{u_k\}_{k=1}^{\infty}$ obtained by Theorem \ref{thm:3.1} are lying in the gap of $0,w_0-v_0$,
by Lemma \ref{lem:subsequ}
we can extract a subsequence converging to a function $u\in\hat{\Gamma}(v_0,w_0)$ such that $u+v_0$ is a solution of \eqref{eq:PDE}.
\begin{cor}
There is a subsequence of $\tau_{-j_{k}}^{1} u_{k}$
converging to a function $U\in \hat{\Gamma}(v_0,w_0)$ with $U(\bd{0})\geq u_0(\bd{0})$ and
$U+v_0$ is a solution of \eqref{eq:PDE}, where $u_0$ is as in Theorem \ref{thm:3.1}.
\end{cor}
\begin{proof}
By Theorem \ref{thm:3.1},
without loss of generality $u_k$ intersects $u_0$ for all $k\in\N$.
Thus we have $\tau_{-j_{k}}^{1} u_{k}\left(\bd{0}\right)>u_{0}\left(\bd{0}\right)$
for some $j_k\in\Z$.
Let $U_{k}:=\tau^{1}_{-j_{k}} u_{k}$,
then $U_{k}+v_0$ are solutions of \eqref{eq:PDE}. Since $0 \leq U_k\leq w_0-v_0$, by Lemma \ref{lem:subsequ},
there is a function $U\in\R^{\Z\times \Z^{n-1}/\Z^{n-1}}$ such that $U_k \to U$ pointwise (up to a subsequence).
So $U(\bd{0})\geq u_{0}(\bd{0})$
and $U+v_0$ is a solution of \eqref{eq:PDE}.
\end{proof}

Since $U+v_0\in \hat{\Gamma}(v_0,w_0)$, either $J_1(U+v_0)=\infty$ or
$J_1(U+v_0)<\infty$.
\begin{enumerate}
  \item \label{item:99} Suppose $J_1(U+v_0)=\infty$. Then the solution $U+v_0$ is different from any known solutions.
For instance, $U+v_0$ is not a minimal and Birkhoff solution in \cite{Miao, LC},
and is not a multitransition solution in \cite{LC1}
since
the functional $J_1$ at any of
these solutions is finite.

  \item \label{item:100} Now assume $J_1(U+v_0)<\infty$.
By
Lemma \ref{lem:6553}, there are
$\phi,\psi\in\{v_0,w_0\}$ satisfying
\begin{equation}\label{eq:0011092}
\begin{aligned}
|(U+v_0-\phi)(\bd{T}_{i})| \rightarrow 0, &\quad i \rightarrow-\infty ; \\
|(U+v_0-\psi)(\bd{T}_{i})| \rightarrow 0, &\quad i \rightarrow \infty.
\end{aligned}
\end{equation}
\begin{enumerate}
\item \label{item:100-1} If $\phi=\psi$, then $U+v_0$ is a homoclinic solution. 
\item \label{item:100-2} Otherwise suppose $\phi\neq \psi$ and assume there is no subsequence of $\tau_{-j_{k}}^{1} u_{k}+v_0$
converging to homoclinic solution.
Without loss of generality, set $\phi=v_0$ and $\psi=w_0$.
Then $U+v_0$ is a heteroclinic solution.
In this case, it is  interesting that we can construct another solution of \eqref{eq:PDE}. 
Indeed,
for $\epsilon>0$, by \eqref{eq:0011092} there exists $i_0(\epsilon)\in\N$ satisfying %
\begin{equation}\label{eq:4.7}
\begin{aligned}
\left|(U+v_0-v_{0})(\bd{T}_{i})\right|  \leq \epsilon, &\quad i \leq-i_{0}(\epsilon), \\
\left|(U+v_0-w_{0})(\bd{T}_{i})\right| \leq \epsilon, & \quad i \geq i_{0}(\epsilon).
\end{aligned}
\end{equation}
Since $U_k\to U$ pointwise
as $k\to\infty$, there is a $k_0=k_0(\epsilon)$ such that
for $k\geq k_0$,
\begin{equation*}
\begin{aligned}
\left|(U_{k}+v_0 -v_{0})(\bd{T}_{-i_{0}(\epsilon)})\right|& \leq 2 \epsilon,\\
\left|(U_{k}+v_0-w_{0})(\bd{T}_{i_{0}(\epsilon)})\right|&\leq 2 \epsilon.
\end{aligned}
\end{equation*}
Thus we have
$U_k(\bd{T}_{-i_0(\epsilon)}) <u_{0}(\bd{T}_{-i_0(\epsilon)})$ and
$u_{0}(\bd{T}_{i_0(\epsilon)})<U_k(\bd{T}_{i_0(\epsilon)})$ provided $\epsilon$ sufficiently small.
But noticing that
$U_k$ is $k$-periodic in $\bd{i}_1$, we obtain
$$
\left|(U_{k}+v_0-v_{0})(\bd{T}_{k-i_{0}(\epsilon)})\right| \leq 2 \epsilon
$$
and $U_k(\bd{T}_{k-i_0(\epsilon)})< u_{0}(\bd{T}_{k-i_0(\epsilon)})$ %
with $k-i_0(\epsilon)>i_0(\epsilon)$.
Hence
there is a $q_k\in(i_0(\epsilon),k-i_0(\epsilon))$ such that
\begin{equation}\label{eq:4.8}
\begin{aligned}
U_{k}(\bd{T}_i) &\geq u_{0}(\bd{T}_i), \quad  i_{0}(\epsilon) \leq i \leq q_{k}-1;\\
 U_{k}(\bd{T}_{q_k}) &< u_{0}(\bd{T}_{q_k}).
\end{aligned}
\end{equation}
Let $W_k:=\tau^{1}_{-q_k}U_k$.
Then proceeding as for $U$, we have that
$W_{k} \rightarrow W$ pointwise
for some $W$ satisfying
$W+v_0\in\hat{\Gamma}_1(v_0,w_0)$
and
$W(\bd{0})\leq u_{0}(\bd{0})$.
If $J_1(W+v_0)=\infty$, then similar to \eqref{item:99}, $W+v_0$ is a new solution which is different with $U+v_0$ since $J_1(U+v_0)<\infty$.
So we assume $J_1(W+v_0)<\infty$.
We claim
\begin{equation}\label{eq:kfls}
  q_k\to \infty \textrm{\quad as \quad} k\to\infty.
\end{equation}
Suppose \eqref{eq:kfls} holds for the moment.
Then
\eqref{eq:4.8} implies
\begin{equation}\label{eq:4.9}
W(\bd{i}) \geq u_{0}(\bd{i}), \quad \bd{i}_{1} < 0.
\end{equation}
Hence applying
Lemma \ref{lem:6553} shows that
\begin{equation*}
|(W+v_0-w_0)(\bd{T}_i)| \rightarrow 0, \quad i \rightarrow-\infty .
\end{equation*}
Since we assume there is no subsequence of $\tau_{-j_{k}}^{1} u_{k}+v_0$
converging to homoclinic solution, so $W+v_0$ is a heteroclinic solution from $w_0$ to $v_0$.

What is left is
to show \eqref{eq:kfls}.
Suppose, by contradiction, that $q_k$ is bounded.
Then up to a subsequence, we can assume $q_k\equiv q > i_0(\epsilon)$ and $W=\tau^{1}_{-q}U$.
Thus
\begin{equation*}
\begin{aligned}
&U(\bd{T}_i) \geq u_{0}(\bd{T}_i), \quad  i_{0}(\epsilon) \leq i \leq q-1;\\
& U(\bd{T}_{q}) \leq u_{0}(\bd{T}_{q}).
\end{aligned}
\end{equation*}
But by \eqref{eq:4.7} for $\epsilon$ small enough, $U(\bd{T}_{q}) > u_{0}(\bd{T}_{q})$, a contradiction.
Thus \eqref{eq:kfls} is proved. 
\end{enumerate}
\end{enumerate}

Summarizing the above discussion, we have the following Table \ref{table:1},
where \eqref{eq:oec} is
\begin{equation}\label{eq:oec}
  \textrm{there is no subsequence of $\tau_{-j_k}^{1}u_k$ converging to homoclinic solution.}
\end{equation}
\begin{table}[ht]
\caption{Limits of $\tau_{-j_k}^{1}u_k$}\label{table:1}
\begin{tabular}{|c|c|c|c|}
  \hline
  $J_1(U+v_0)$             & \multicolumn{3}{|c|}{$U+v_0$} \\ \hline
  \eqref{item:99}: $=\infty$                & \multicolumn{3}{|c|}{A new solution}  \\ \hline
   \multirow{5}*{\eqref{item:100}: $<\infty$} & \multicolumn{3}{|c|}{\eqref{item:100-1}: Homoclinic solution} \\ \cline{2-4}
                           &\multirow{4}{38mm}{\eqref{item:100-2}: Heteroclinic solution (say, from $v_0$ to $w_0$) and \eqref{eq:oec} holds}   &  $J_1(V+v_0)$             & $V+v_0$   \\ \cline{3-4}
                           &                    &  $=\infty$                & A new solution\\ \cline{3-4}
                           &              &  \multirow{2}*{$<\infty$} & Heteroclinic solution \\
                           &                                       &                           & (from $w_0$ to $v_0$) \\  \hline
\end{tabular}
\end{table}

\section{Mountain pass solutions in the gap of $\MM_1(v_0,w_0)$}\label{sec:2006}

We construct heteroclinic mountain pass solution in this section.
The difference is that the sum in the definition of $I_{0}^{\bd{p}}$ of \eqref{eq:1.2} involves only finite many terms,
but in this section
a new functional, $I_{1}^{\bd{q}}$, will be a sum of infinitely many terms.
Suppose that \eqref{eq:*0} and \eqref{eq:*1} hold.
For $\bd{q}=(\bd{q}_2,\cdots,\bd{q}_n)\in\N^{n-1}$,
set
\begin{equation}\label{eq:gamma1865411}
\begin{split}
\Lambda_{1}^{\bd{q}}=\{u\in\R^{\Z\times \Z^{n-1} /(\bd{q}\Z^{n-1})}\,|\,
\norm{u}_{\Lambda_{1}^{\bd{q}}}:=&\norm{u}_{\ell^1}+\norm{u}_{\ell^2}\\
:=&\sum_{\bd{j}\in \Z\times \T_{1}^{\bd{q}}}|u(\bd{j})| + \sqrt{\sum_{\bd{j}\in \Z\times \T_{1}^{\bd{q}}}|u(\bd{j})|^{2}} <\infty
\}.
\end{split}
\end{equation}
Obviously $(\Lambda_{1}^{\bd{q}},\norm{\cdot}_{\Lambda_{1}^{\bd{q}}})$ is a Banach space with norm $\norm{\cdot}_{\Lambda_{1}^{\bd{q}}}$.
In fact, this norm is $\norm{\cdot}_{\ell^{1}(\Z\times \T_{1}^{\bd{q}})}+\norm{\cdot}_{\ell^{2}(\Z\times \T_{1}^{\bd{q}})}$ on $\ell^{1}(\Z\times \T_{1}^{\bd{q}})\cap \ell^{2}(\Z\times \T_{1}^{\bd{q}})$.
\begin{rem}
The norm $\|\cdot\|_{\ell^2}$ will be used in a similar result of Lemma \ref{lem:2.2} \eqref{lem:2.2-5} of
the deformation lemma.
In the proof of Lemma \ref{lem:2.2} \eqref{lem:2.2-5},
$\|\cdot\|_{\ell^2}$ can be replaced by equivalent norm $\|\cdot\|_{\ell^1}$ since there are only finite many terms.
But for infinitely many term in the definition of $I_{1}^{\bd{q}}$, $\|\cdot\|_{\ell^2}$ cannot be replaced by $\|\cdot\|_{\ell^1}$ any more.
Noticing that to show $I_{1}^{\bd{q}}$ is $C^1$, one need the norm $\|\cdot\|_{\ell^1}$,
which cannot be replaced by $\|\cdot\|_{\ell^2}$.
Please see the proof of Proposition \ref{prop:diff-565}.
\end{rem}
Note that
\begin{equation}\label{eq:04272}
w_1(\bd{i})-v_1(\bd{i})\leq w_0(\bd{0})-v_0(\bd{0})\leq 1
\end{equation}
and $$\Big[v_1(\bd{i}),w_1(\bd{i})\Big)\cap \Big[v_1(\bd{j}),w_1(\bd{j})\Big)=\emptyset \quad\textrm{for } \bd{i}_1\neq \bd{j}_1,$$ so
\begin{equation*}
\sum_{\bd{j}\in \Z\times \T_{1}^{\bd{q}}}|w_1(\bd{j})-v_1(\bd{j})|\leq C(\bd{q})
\end{equation*}
for some constant $C(\bd{q})$ depending only on $\bd{q}$.
Thus by \eqref{eq:04272},
\begin{equation*}
\sum_{\bd{j}\in \Z\times \T_{1}^{\bd{q}}}|w_1(\bd{j})-v_1(\bd{j})|^2
\leq\sum_{\bd{j}\in \Z\times \T_{1}^{\bd{q}}}|w_1(\bd{j})-v_1(\bd{j})|
\leq C(\bd{q})
\end{equation*}
and hence $\norm{w_1-v_1}_{\Lambda_{1}^{\bd{q}}}<\infty$, i.e., $w_1-v_1\in \Lambda_{1}^{\bd{q}}$.

For $u\in \Lambda_{1}^{\bd{q}}$, define
\begin{equation*}
I_{1;p,q}^{\bd{q}}(u):=J_{1;p,q}^{\bd{q}}(u+v_1),
\end{equation*}
where $J_{1;p,q}^{\bd{q}}$ is defined in Section \ref{sec:312432}.
For simplicity, set $c_1=c_1(v_0,w_0)$ and $c_{1}^{\bd{q}}=c_{1}^{\bd{q}}(v_0,w_0)$.
Since $u\in \Lambda_{1}^{\bd{q}}$ implies
$|u(\bd{i})|\leq j_0$
for some $j_0=j_0(u)\in \N$, we have
$v_0 -j_0 \leq u+v_1\leq w_0+j_0$ and
\begin{equation*}
  \begin{split}
    \lim_{i\to -\infty}\sum_{\bd{j}\in \{i\}\times \T^{\bd{q}}_{1}}
|(u+v_1)(\bd{j})-v_0(\bd{j})|&= 0,\\
     \lim_{i\to \infty}\sum_{\bd{j}\in \{i\}\times \T^{\bd{q}}_{1}}
|(u+v_1)(\bd{j})-w_0(\bd{j})|&= 0.
  \end{split}
\end{equation*}
Thus by Lemma \ref{prop:8554556} and Remark \ref{rem:000123123},
\begin{equation*}
I_{1}^{\bd{q}}(u):=\liminf_{p\to -\infty \atop q\to \infty}I_{1;p,q}^{\bd{q}}(u)
\end{equation*}
is well-defined and
if $I_{1}^{\bd{q}}(u)<\infty$, then
\begin{equation}\label{eq:9658896}
  I_{1}^{\bd{q}}(u)=\lim_{p\to -\infty \atop q\to \infty}I_{1;p,q}^{\bd{q}}(u), \quad \textrm{ i.e., }\quad
I_{1}^{\bd{q}}(u)=\sum_{i\in\Z}I_{1,i}^{\bd{q}}(u).
\end{equation}
Since we use a modified Mountain Pass Theorem to show the existence of critical point,
the functional should be well-defined from $\Lambda_{1}^{\bd{q}}$
to $\R$ and be $C^1$.
Fortunately, this is the case, as the following two propositions show.
\begin{prop}\label{prop:90999}
For any $u\in\Lambda_{1}^{\bd{q}}$,
$I_{1}^{\bd{q}}(u)<\infty$ and thus
\eqref{eq:9658896} holds.
\end{prop}
\begin{proof}
Assume $u\in\Lambda_{1}^{\bd{q}}$, then there exists $j_0=j_0(u)\in \N$ such that
$v_0-j_0\leq v_1,v_1+u\leq w_0+j_0$.
Thus $v_1, v_1+u$ have bounded action.
By Lemma \ref{lem:lc16321},
there exists some $L=L(u,r)(=L(j_0, r))$, such that
\begin{equation*}
|I_{1;p,q}^{\bd{q}}(u)-I_{1;p,q}^{\bd{q}}(0)|\leq
L\sum_{\bd{j}\in \overline{[p,q]\times \T_{1}^{\bd{q}}}}|u(\bd{j})|\leq
L \cdot C(r)\norm{u}_{\Lambda_{1}^{\bd{q}}},
\end{equation*}
where $C(r)$ is a constant depending on $r$.
Thus
\begin{equation}\label{eq:need}
\begin{split}
  I_{1;p,q}^{\bd{q}}(u)\leq &I_{1;p,q}^{\bd{q}}(0) +L \cdot C(r)\norm{u}_{\Lambda_{1}^{\bd{q}}}\\
  \leq &(I_{1}^{\bd{q}}(0) +2K_1)+L \cdot C(r)\norm{u}_{\Lambda_{1}^{\bd{q}}}\\
  = &(c_{1}^{\bd{q}}+2K_1)+L \cdot C(r)\norm{u}_{\Lambda_{1}^{\bd{q}}}\\
  <&\infty,
  \end{split}
\end{equation}
where the second inequality follows from \eqref{eq:pppp5} with $K_1=K_1(\bd{q}, u, v_0,w_0)$.
Then $I_{1}^{\bd{q}}(u)<\infty$ and thus \eqref{eq:9658896}
follows.
\end{proof}

\begin{prop}\label{prop:diff-565}
We have
$I_{1}^{\bd{q}}\in C^1(\Lambda_{1}^{\bd{q}},\R)$.
If $(I_{1}^{\bd{q}})'(u)=0$, i.e., $u$ is a critical point of $I_{1}^{\bd{q}}$,
then $u+v_1$ is a solution of \eqref{eq:PDE}.
\end{prop}

\begin{proof}
Firstly, we prove that $I_{1}^{\bd{q}}$ is Gateaux differentiable.
For $u, v\in \Lambda_{1}^{\bd{q}}$, and $t\in[-1,1]\setminus \{0\}$,
\begin{equation*}
     \Big|\sum_{\bd{j}\in \{i\}\times \T_{1}^{\bd{q}}} \frac{S_{\bd{j}}(u+tv+v_1)-S_{\bd{j}}(u+v_1)}{t}\Big|
    \leq L\sum_{\bd{j}\in \overline{\{i\}\times \T_{1}^{\bd{q}}}} |v(\bd{j})|
    =:M_i.
\end{equation*}
Since $v\in\Lambda_{1}^{\bd{q}}$, $\sum_{i\in\Z}M_i<\infty$.
Thus we have
\begin{equation}\label{eq:04271}
\begin{split}
  (I_{1}^{\bd{q}})'(u)v
  =&\sum_{i\in \Z}\sum_{\bd{j}\in \{i\}\times \T_{1}^{\bd{q}}} \sum_{\bd{k}: \norm{\bd{k}-\bd{j}}\leq r}\partial_{\bd{k}}S_{\bd{j}}(u+v_1)\cdot v(\bd{k})\\
  =&\sum_{i\in \Z}\sum_{\bd{j}\in \{i\}\times \T_{1}^{\bd{q}}} v(\bd{j})\sum_{\bd{k}: \norm{\bd{k}-\bd{j}}\leq r}\partial_{\bd{j}}S_{\bd{k}}(u+v_1).
    \end{split}
\end{equation}

To show that $(I_{1}^{\bd{q}})'$ is continuous, set $u_m\to u$ in $\Lambda_{1}^{\bd{q}}$ as $m\to \infty$. Then
\begin{equation}\label{eq:0427}
  \begin{split}
    &|[(I_{1}^{\bd{q}})'(u_m)-(I_{1}^{\bd{q}})'(u)]v| \\
    =& \Big| \sum_{i\in\Z} \sum_{\bd{j}\in \{i\}\times \T_{1}^{\bd{q}}} v(\bd{j})\sum_{\bd{k}: \norm{\bd{k}-\bd{j}}\leq r}[\partial_{\bd{j}}S_{\bd{k}}(u_m+v_1)- \partial_{\bd{j}}S_{\bd{k}}(u+v_1)]\Big| \\
      =& \Big| \sum_{i\in\Z} \sum_{\bd{j}\in \{i\}\times \T_{1}^{\bd{q}}} v(\bd{j})\sum_{\bd{k}: \norm{\bd{k}-\bd{j}}\leq r}
      \int_{0}^{1}\frac{\ud}{\ud t}\partial_{\bd{j}}S_{\bd{k}}(u+t(u_m-u)+v_1)\ud t \Big|\\
      =& \Big| \sum_{i\in\Z} \sum_{\bd{j}\in \{i\}\times \T_{1}^{\bd{q}}} v(\bd{j})\sum_{\bd{k}: \norm{\bd{k}-\bd{j}}\leq r}
      \int_{0}^{1}\sum_{\bd{l}: \norm{\bd{l}-\bd{k}}\leq r}\partial_{\bd{j},\bd{l}}S_{\bd{k}}(u+t(u_m-u)+v_1)\ud t\, \cdot(u_m-u)(\bd{l}) \Big|\\
      \leq & C \sum_{i\in\Z} \sum_{\bd{j}\in \{i\}\times \T_{1}^{\bd{q}}} |v(\bd{j})|\sum_{\bd{k}: \norm{\bd{k}-\bd{j}}\leq r}\sum_{\bd{l}: \norm{\bd{l}-\bd{k}}\leq r}
     |(u_m-u)(\bd{l})|.
  \end{split}
\end{equation}
So $(I_{1}^{\bd{q}})'(u_m)\to (I_{1}^{\bd{q}})'(u)$ as $m\to\infty$.
By \eqref{eq:04271}, if $u$ is a critical point of $I_{1}^{\bd{q}}$, then $u+v_1$ is a solution of \eqref{eq:PDE}.
So the proof of Proposition \ref{prop:diff-565} is complete.
\end{proof}
Now following Section \ref{sec:mountainpass}, let us
define
semiflow $\Phi_{t}^{1}: \Lambda_{1}^{\bd{q}}\to \Lambda_{1}^{\bd{q}}$ as follows:
\begin{equation*}
  \left\{
    \begin{array}{ll}
      -\partial_t \Phi_{t}^{1} (u)(\bd{i}) & =\sum_{\bd{j}: \norm{\bd{j}-\bd{i}}\leq r}\partial_{\bd{i}}S_{\bd{j}}(\Phi_{t}^{1} (u)+v_1), \quad \quad \textrm{for }t>0,\\
      \Phi_{0}^{1} (u)(\bd{i}) & =u(\bd{i}).
    \end{array}
  \right.
\end{equation*}
Set $W(u)(\bd{i}):=\sum_{\bd{j}: \norm{\bd{j}-\bd{i}}\leq r}\partial_{\bd{i}}S_{\bd{j}}(u+v_1)$.
Similar proof of \eqref{eq:lipschitz} shows that $\|W(u)-W(v)\|_{\ell^2}\leq C\|u-v\|_{\ell^2}$ for some $C=C(r)$.
The proof of $\|W(u)-W(v)\|_{\ell^1}\leq C\|u-v\|_{\ell^1}$ is easier (cf. \eqref{eq:0427}).
So $\Phi_{t}^{1}$ is well-defined and is $C^1$ in $t$.
Moreover, a new version of Proposition \ref{prop:comparison} is obtained.

Set
\begin{equation*}
\MG^{\bd{q}}_{1} =\left\{u \in \Lambda_{1}^{\bd{q}} \, |\, 0 \leq u \leq w_{1}-v_{1}\right\}.
\end{equation*}
Note that if $u\in \R^{\Z\times \Z^{n-1} /(\bd{q}\Z^{n-1})}$ and $0\leq u\leq w_1-v_1$,
\begin{equation}\label{eq:gamma1p}
  \norm{u}_{\Lambda_{1}^{\bd{q}}}\leq \norm{w_1-v_1}_{\Lambda_{1}^{\bd{q}}}<\infty.
\end{equation}
In other words, $\{u\in\R^{\Z\times \Z^{n-1} /(\bd{q}\Z^{n-1})}\,|\, 0\leq u\leq w_1-v_1\}\subset \MG^{\bd{q}}_{1}$.
\begin{prop}\label{prop:compact1}
$\MG_{1}^{\bd{q}}$ is compact with respect to the norm $\norm{\cdot}_{\Lambda_{1}^{\bd{q}}}$.
\end{prop}
\begin{proof}
For any $u_n\in \MG_{1}^{\bd{q}}$, by Lemma \ref{lem:subsequ} and \eqref{eq:gamma1p}, there is a $u\in \MG_{1}^{\bd{q}}$
such that $u_n\to u$ (maybe up to a subsequence) pointewise as $n\to \infty$ and
$|(u_n-u)(\bd{j})|\leq (w_1-v_1)(\bd{j}) $.
Note that
\begin{equation*}
\begin{split}
  \norm{u_k-u}_{\Lambda_{1}^{\bd{q}}}
  =&\sum_{\bd{j}\in [-N,N]\times \T_{1}^{\bd{q}}}|(u_k-u)(\bd{j})|
  +\sum_{\bd{j}\in (\Z\setminus [-N,N])\times \T_{1}^{\bd{q}}}|(u_k-u)(\bd{j})|\\
 \quad& + \left(\sum_{\bd{j}\in [-N,N]\times \T_{1}^{\bd{q}}}|(u_k-u)(\bd{j})|^2
  +\sum_{\bd{j}\in (\Z\setminus [-N,N])\times \T_{1}^{\bd{q}}}|(u_k-u)(\bd{j})|^2\right)^{\frac{1}{2}}\\
  \leq &\sum_{\bd{j}\in [-N,N]\times \T_{1}^{\bd{q}}}|(u_k-u)(\bd{j})|
  +\sum_{\bd{j}\in (\Z\setminus [-N,N])\times \T_{1}^{\bd{q}}}|(u_k-u)(\bd{j})|\\
 \quad& + \left(\sum_{\bd{j}\in [-N,N]\times \T_{1}^{\bd{q}}}|(u_k-u)(\bd{j})|^2\right)^{\frac{1}{2}}
  +\left(\sum_{\bd{j}\in (\Z\setminus [-N,N])\times \T_{1}^{\bd{q}}}|(u_k-u)(\bd{j})|^2\right)^{\frac{1}{2}}
  \end{split}
\end{equation*}
and
\begin{equation*}
\begin{split}
  &\sum_{\bd{j}\in (\Z\setminus [-N,N])\times \T_{1}^{\bd{q}}}|(u_k-u)(\bd{j})|+
  \left(\sum_{\bd{j}\in (\Z\setminus [-N,N])\times \T_{1}^{\bd{q}}}|(u_k-u)(\bd{j})|^2\right)^{\frac{1}{2}}\\
  \leq&
  \sum_{\bd{j}\in (\Z\setminus [-N,N])\times \T_{1}^{\bd{q}}}(w_1-v_1)(\bd{j})+
  \left(\sum_{\bd{j}\in (\Z\setminus [-N,N])\times \T_{1}^{\bd{q}}}|(w_1-v_1)(\bd{j})|^2\right)^{\frac{1}{2}}\\
  \to &
  0
  \end{split}
\end{equation*}
as $N\to \infty$.
Thus for any $\epsilon>0$, one can choose $N$ sufficiently large such that
 $$ \sum_{\bd{j}\in (\Z\setminus [-N,N])\times \T_{1}^{\bd{q}}}|(u_k-u)(\bd{j})|+
  \left(\sum_{\bd{j}\in (\Z\setminus [-N,N])\times \T_{1}^{\bd{q}}}|(u_k-u)(\bd{j})|^2\right)^{\frac{1}{2}}
  \leq \epsilon,$$
  so
\begin{equation*}
\begin{split}
&\lim_{k\to\infty}\norm{u_k-u}_{\Lambda_{1}^{\bd{q}}}\\
\leq&
\lim_{k\to\infty}\sum_{\bd{j}\in [-N,N]\times \T_{1}^{\bd{q}}}|(u_k-u)(\bd{j})|+
  \left(\sum_{\bd{j}\in [-N,N]\times \T_{1}^{\bd{q}}}|(u_k-u)(\bd{j})|^2\right)^{\frac{1}{2}}+\epsilon\\
=&\epsilon.
\end{split}
\end{equation*}
Since $\epsilon$ is arbitrary, Proposition \ref{prop:compact1} is proved.
\end{proof}
Set
\begin{equation*}
  \MH^{\bd{q}}_{1} =\left\{h \in C\left([0,1], \MG^{\bd{q}}_{1}\right) \,|\, h(0)=0, h(1)=w_{1}-v_{1}\right\}.
\end{equation*}
Hence one have new versions of Lemma \ref{lem:2.2}, Proposition \ref{thm:2.1-1-99} and Theorem \ref{thm:2.1}.
Thus we obtain a mountain pass critical point $u_{\bd{q}}$ and then a heteroclinic mountain pass solution $u_{\bd{q}} +v_1$
satisfying $v_1<u_{\bd{q}}+v_1<w_1$
for any $\bd{q}\in\N^{n-1}$.
Next we study the multiplicity of heteroclinic mountain pass solutions.
It sufficies to prove a silimlar result of Proposition \ref{prop:2.3}.
To this end, for each $k\in\N$, let $\bd{q}(k)=(k,1,\cdots, 1)\in\N^{n-1}$.

Set
$$
d_{1}^{\bd{q}(k)}=\inf _{h \in \MH_{1}^{\bd{q}(k)}} \max _{\theta \in[0,1]} I_{1}^{\bd{q}(k)}(h(\theta)).
$$
Then $d_{1}^{\bd{q}(k)}$ is a mountain pass
critical value of $I_1 ^{\bd{q}(k)}$ on $\MG_{1}^{\bd{q}(k)}$ with a corresponding mountain pass critical point $U_k$ such that $0 < U_k < w_1-v_1$ with $I_{1}^{\bd{q}(k)}(U_k) = d_{1}^{\bd{q}(k)} > c_{1}^{\bd{q}(k)}$ and $U_k+v_1$ is a solution of \eqref{eq:PDE}.
We have:
\begin{prop}\label{prop:5.1}
There is a constant $M_1$, independent of $k$, such that
$$0<d_{1}^{\bd{q}(k)}-c_{1}^{\bd{q}(k)}\leq M_1.$$
\end{prop}
\begin{proof}
For $u\in \Lambda_{1}^{\bd{q}(k)}$, let $I_1(u):=\sum_{\bd{j}\in\Z\times\{0\}^{n-1}}[S_{\bd{j}}(u+v_1)-c_0]$, then
$$
I_{1}^{\bd{q}(k)}(u)=\sum_{i=0}^{k-1} I_{1}(\tau^{2}_{-i} u).
$$
Set
$h_{k}(\theta)=\phi_{k}\left(\theta,\bd{i}_{2}\right)\left(w_{1}-v_{1}\right),$
where $\phi_k$ is defined in \eqref{eq:phik}.
To prove Proposition \ref{prop:5.1},
it suffices to show that
\begin{equation}\label{eq:5.3}
\max_{\theta\in[0,1]}I_{1}^{\bd{q}(k)}\left(h_{k}(\theta)\right) -c_{1}^{\bd{q}(k)}\leq M_{1}
\end{equation}
holds for some $M_1$ independent of $k$.
Note
$$
\begin{aligned}
I_{1}^{\bd{q}(k)}\left(h_{k}(\theta)\right)
&=\sum_{\bd{j}\in\Z\times [0,k-1]\times \{0\}^{n-2}}[S_{\bd{j}}(h_k(\theta)+v_1)-c_0]\\
&=\sum_{\bd{j}\in\Z\times [a,a+k-1]\times \{0\}^{n-2}}[S_{\bd{j}}(h_k(\theta)+v_1)-c_0]
\end{aligned}
$$
for any $a\in\Z$.
For any $\theta\in [0,1]$, let
\begin{equation*}
\begin{split}
\MA:=\{\bd{i}\in\Z^n \,|\, &0\leq \bd{i}_2\leq k-1, \bd{i}_3=\cdots=\bd{i}_n=0,\\
&\phi_{k}\left(\theta,\bd{i}_{2}\right) \neq 0 \text { or } \phi_{k}\left(\theta,\bd{i}_{2}\right) \neq 1\}.
\end{split}
\end{equation*}
By the construction of $\phi_k$, $\MA$
consists of at most two regions, say $R_i$ ($i=1,2$) of the form
$R_{i}=\mathbb{Z} \times \{a_{i}\} \times \{0\}^{n-2}$ with $a_1\leq a_2$.
Therefore
\begin{equation}\label{eq:04274}
\textrm{$h_k=0$ or $w_1-v_1$ on
$\bigcup_{i=0}^{k-1} (\Z\times \{i\}\times \{0\}^{n-2} ) \backslash\left(R_{1} \cup R_{2}\right)$}.
\end{equation}
Then
\begin{equation}\label{eq:04273}
\begin{split}
  &I_{1}^{\bd{q}(k)}\left(h_{k}(\theta)\right)-c_{1}^{\bd{q}(k)} \\
  =&  \sum_{i\in [a, a+k-1]}\left\{\sum_{\bd{j}\in\Z\times \{i\} \times\{0\}}[S_{\bd{j}}(h_k(\theta)+v_1)-c_0]-c_1\right\}\\
    =& \sum_{i\in [a, a+k-1]} \left\{\sum_{\bd{j}\in\Z\times \{i\}\times\{0\}}[S_{\bd{j}}(h_k(\theta)+v_1)-c_0]-
    \sum_{\bd{j}\in\Z\times \{i\}\times\{0\}}[S_{\bd{j}}(v_1)-c_0]\right\}\\
    =&\sum_{i\in [a, a+k-1]}\sum_{\bd{j}\in\Z\times \{i\}\times\{0\}}[S_{\bd{j}}(h_k(\theta)+v_1)-S_{\bd{j}}(v_1)].
\end{split}
\end{equation}
Then by \eqref{eq:04274}, the cardinality of
$$\{i\in \mathcal{B}\,|\,\sum_{\bd{j}\in\Z\times \{i\}\times\{0\}}S_{\bd{j}}(h_k(\theta)+v_1)-S_{\bd{j}}(v_1)\neq 0
\}
$$
is at most a finite number, denoted by $C_1$, independent of $k$, where
\begin{equation*}
  \mathcal{B}:=\left\{
                 \begin{array}{ll}
                   {\cup_{i=a_1-r}^{a_1-r+k-1}\{i\}}\setminus {\cup_{i=a_1-r}^{a_2+r}\{i\}}, & \textrm{if $a_2-a_1\leq 2r$;} \\
                   {\cup_{i=a_1-r}^{a_1-r+k-1}\{i\}}\setminus {(\cup_{i=a_1-r}^{a_1+r}\{i\}\cup\cup_{i=a_2-r}^{a_2+r} \{i\})}, & \textrm{if $2r <a_2- a_1$ and $a_1+k-a_2> 2r$;}\\
                   {\cup_{i=a_2-r}^{a_2-r+k-1}\{i\}}\setminus {\cup_{i=a_2-r}^{a_1+r+k}\{i\}}, & \textrm{if $2r <a_2- a_1$ and $a_1+k-a_2\leq 2r$.}
                 \end{array}
               \right.
\end{equation*}

Noticing that for any $i\in\Z$,
\begin{equation}\label{eq:0428}
\begin{split}
&|\sum_{\bd{j}\in\Z\times \{i\}\times\{0\}}[I_{1}(h_k(\theta))-c_1]|\\
 = &|\sum_{\bd{j}\in\Z\times \{i\}\times\{0\}}S_{\bd{j}}(h_k(\theta)+v_1)-S_{\bd{j}}(v_1)| \\
   = & |\sum_{\bd{j}\in\Z\times \{i\}\times\{0\}}\int_{0}^{1} \sum_{\bd{k}:\norm{\bd{k}-\bd{j}}\leq r} \partial_{\bd{k}}S_{\bd{j}}(v_1 +t h_k(\theta))\ud t \cdot h_k(\theta)(\bd{k})|\\
   \leq &L\sum_{\bd{k}\in\overline{\Z\times \{i\}\times\{0\}}}|h_k(\theta)(\bd{k})|\\
   \leq & L \sum_{\bd{k}\in\overline{\Z\times \{i\}\times\{0\}}}[w_1(\bd{k})-v_1(\bd{k})]\\
   \leq & L C\norm{w_1-v_1}_{\Lambda_{1}^{\bd{q}}}\\
   =:&M.
\end{split}
\end{equation}
Here $L=L(w_0-v_0, r)$, $C=C(r)$ are constants and
$\overline{\Z\times \{i\}\times\{0\}}=\cup_{\bd{j}\in \Z\times \{i\}\times\{0\}}\{\bd{k}\in\Z^n\,|\, \norm{\bd{k}-\bd{j}}\leq r\}$.
The first inequality in \eqref{eq:0428} needs to be explained.
Since $v_1 +t h_k(\theta)$ have bounded action with bounded constant $w_0-v_0$ for all $t\in [0,1]$, thus by (S\ref{eq:S1}), there exists $L=L(w_0-v_0, r)$ such that the first inequality in \eqref{eq:0428} holds (cf. \cite[the proof of Lemma 2.4]{Miao}).

By \eqref{eq:need} and \eqref{eq:gamma1p}, 
there is an $M(\bd{q})>0$ such that $I_{1}^{\bd{q}}(u)-c_{1}^{\bd{q}}\leq M(\bd{q})$ for all $u\in \MG_{1}^{\bd{q}}.$
So without loss of generality, assume $k> 4r+1$ and $M\geq \max(M(\bd{q}(1)),\cdots, M(\bd{q}(4r)))$.
We have the following three cases. 
\begin{itemize}
  \item If $a_2- a_1\leq 2r$,
  \begin{equation*}
  \begin{split}
  &|I_{1}^{\bd{q}(k)}(u)-kc_1|\\
  =&|\sum_{i=a_1-r}^{a_2+r}[I_{1}(\tau^{2}_{-i}u)-c_1]+\sum_{i\in\mathcal{B}}[I_{1}(\tau^{2}_{-i}u)-c_1]|\\
  \leq& M\cdot(a_2+2r+1-a_1)+M\cdot C_1\\
  \leq& M\cdot (4r+1+C_1).
  \end{split}
  \end{equation*}
  \item If $2r <a_2- a_1$ and $a_1+k-a_2> 2r$, i.e., $a_1\not\in (a_2-2r,a_2+2r)$,
  \begin{equation*}
  \begin{split}
  &|I_{1}^{\bd{q}(k)}(u)-kc_1|\\
  =&|\sum_{i=a_1-r}^{a_1+r}[I_{1}(\tau_{-i}^{2}u)-c_1]+\sum_{i=a_2-r}^{a_2+r}[I_{1}(\tau_{-i}^{2}u)-c_1]+\sum_{i\in\mathcal{B}}[I_{1}(\tau^{2}_{-i}u)-c_1]|\\
  \leq& 2M\cdot(2r+1)+M\cdot C_1\\
   =& 2M\cdot(2r+1+C_1).
  \end{split}
  \end{equation*}
  \item If $2r <a_2- a_1$ and $a_1+k-a_2\leq 2r$,
  \begin{equation*}
  \begin{split}
  &|I_{1}^{\bd{q}(k)}(u)-kc_1|\\
  =&|\sum_{i=a_2-r}^{a_1+k+r}[I_{1}(\tau_{-i}^{2}u)-c_1]+\sum_{i\in\mathcal{B}}[I_{1}(\tau^{2}_{-i}u)-c_1]|\\
  \leq &M\cdot(a_1+k+2r+1-a_2)+M\cdot C_1\\
  \leq &M\cdot (4r+1+C_1).
  \end{split}
  \end{equation*}
\end{itemize}
Thus \eqref{eq:5.3} follows by setting $M_1=(4r+2+2C_1)M$.
\end{proof}

Proceeding as in Theorem \ref{thm:2.4} we obtain infinitely many heteroclinic mountain pass solutions.
When we want to go further as in Section \ref{sec:mountainpass} to
see that heteroclinic mountain pass solutions do not constitute a foliation or laminaion,
we encounter more difficulties.
For instance, in the definition of $\underline{\theta}_t$, $\Phi_{t}^{0}h(\theta)\leq u_0$ means either
$\Phi_{t}^{0}h(\theta)$ touches $u_0$ from below or
$\Phi_{t}^{0}h(\theta)=u_0$.
But for unbounded domain $\Z\times\T^{\bd{q}}_{1}$, besides the above two possibilities,
$\Phi_{t}^{1}h(\theta)\leq u_0\in \Lambda_{1}^{\bd{q}}$ may lead to $\Phi_{t}^{1}h(\theta)< u_0$ and
\begin{equation}\label{eq:85698569}
  \Phi_{t}^{1}h(\theta)(\bd{i})\to u_0(\bd{i}) \quad\textrm{as   } |\bd{i}_1| \to 0.
\end{equation}
Notice that \eqref{eq:85698569} always holds since $(w_1-v_1)(\bd{i})\to 0$ as $|\bd{i}_1| \to 0$.
Thus if one want to
show a result similar to Theorem \ref{thm:3.1},
one need a new idea to exclude the third possibility in the proof of Theorem \ref{thm:3.1}.

\begin{rem}
We can obtain another heteroclinic mountain pass solution lying in the gap of $\MM_1(w_0,v_0)$ (please see \cite{LC} for the definition) and
we can construct more heteroclinic mountain pass solutions by the methods of Sections \ref{sec:mountainpass} and \ref{sec:2006} for higher dimension (cf. \cite[Section 5]{LC}).
Thus in the gap of the second laminations (in the sense of \cite{Miao}) we have heteroclinic mountain pass solutions.
\end{rem}

\begin{rem}
Throughout this paper, only the minimal and Birkhoff solutions corresponding to rotation vector $\bd{0}$ are considered (for the definition of rotation vector, please see \cite{LC}).
One can generalize the above results to minimal and Birkhoff solutions corresponding to rotation vector $\alpha\in\Q^{n}$ and obtain corresponding mountain pass solutions.
\end{rem}

\appendix
\section{Appendix}\label{sec:app1}

In this appendix, we prove Proposition \ref{prop:comparison}, Lemma \ref{lem:2.2} and Theorem \ref{thm:2.1} of Sections \ref{sec:mountainpass}.
First is the proof of
Proposition \ref{prop:comparison}, which follows
\cite[Theorem 6.2]{mramor} with slight modifications.

\noindent
\textbf{\emph{Proof of Proposition \ref{prop:comparison}.}}
Define $v(t):=\Phi^{0}_t (u_2) -\Phi^{0}_t (u_1)$.
So $v(0)\geq 0$, $v(0)\neq 0$ and $v$ satisfies the following linear ODE:
\begin{equation*}
\begin{aligned}
&\dot{v}(t)(\bd{i})\\
=&-W(\Phi^{0}_t (u_2))(\bd{i})+W(\Phi^{0}_t (u_1))(\bd{i})\\
=&\int_{0}^{1} \frac{\ud}{\ud \tilde{t}}\left(-\sum_{\bd{j}:\|\bd{i}-\bd{j}\| \leq r} \partial_{\bd{i}} S_{\bd{j}}(\Phi^{0}_t(u_1)+\tilde{t}(\Phi^{0}_t(u_2)-\Phi^{0}_t(u_1))+v_0)\right) \ud \tilde{t} \\
=&\sum_{\bd{j}:\|\bd{i}-\bd{j}\| \leq r}\sum_{\bd{k}:\|\bd{j}-\bd{k}\| \leq r}\left(\int_{0}^{1}-\partial_{\bd{i}, \bd{k}} S_{\bd{j}}(\Phi^{0}_t(u_1)+\tilde{t}(\Phi^{0}_t(u_2)-\Phi^{0}_t(u_1))+v_0) \ud \tilde{t}\right) v(t)(\bd{k})\\
=&:(H(t) v(t))(\bd{i}).
\end{aligned}
\end{equation*}
Similar calculation of \eqref{eq:lipschitz} implies $H ( t ): \Lambda_{0}^{\bd{p}}\to \Lambda_{0}^{\bd{p}}$ is
Lipschitz.
By (S\ref{eq:S3})-(S\ref{eq:S4}),
there is an $M > 0$ such
that the operators
$\tilde{H}(t):=H(t)+M \cdot\mathrm{Id}: \mathbb{R}^{\Z^d} \rightarrow \mathbb{R}^{\Z^{d}}$
are positive: $v\geq 0$ implies $\tilde{H}(t) v \geq 0$.

Note moreover that both the $H(t)$ and the $\tilde{H}(t)$ are uniformly bounded operators, whence the
ODEs
$\dot{v}=H(t) v$
and
$\dot{w}=\tilde{H}(t) w$
define well-posed initial value problems. More importantly, $v(t)$ solves
$\dot{v}=H(t) v$ if and only if $w(t):=e^{M t} v(t)$
solves
$\dot{w}=\tilde{H}(t) w$.
We will now prove that for every $t > 0$ and
every $\bd{i}$, $w(t)(\bd{i})>0$.
Then, obviously, $v(t)(\bd{i}) > 0$ as well, which then
completes the proof of Proposition \ref{prop:comparison}.

To prove the claim on $w(t)$, we solve the initial value problem
for
$\dot{w}=\tilde{H}(t) w$
by Picard iteration, that is we write
\begin{equation*}
w(t)=\left(\sum_{n=0}^{\infty} \tilde{H}^{(n)}(t)\right) w(0),
\end{equation*}
where the $\tilde{H}^{(n)}(t)$ are defined inductively by
\begin{equation*}
\tilde{H}^{(0)}(t)=id \quad \text { and } \quad \tilde{H}^{(n)}(t):=\int_{0}^{t} \tilde{H}(\tilde{t}) \circ \tilde{H}^{(n-1)}(\tilde{t}) d \tilde{t} \quad \text { for } n \geq 1.
\end{equation*}
Observe that the positivity of $\tilde{H}(t)$
implies that the
$\tilde{H}^{(n)}(t)$
are positive as well. Because $w(0) =v(0) \geq 0$ and $v(0)\neq 0$,
we can therefore estimate, for any $\bd{i},\bd{k}\in\Z^d$ with $\norm{\bd{i}-\bd{k}}=1$,
\begin{equation}\label{eq:dnedddd}
\begin{aligned}
w(t)(\bd{i})
&=\left(\sum_{n=0}^{\infty} \tilde{H}^{(n)}(t) w(0)\right)(\bd{i}) \geq\left(\int_{0}^{t}\tilde{H}(\bar{t})w(0)\ud \bar{t}\right) (\bd{i}) \\
& \geq\left(\int_{0}^{t} \int_{0}^{1}-\partial_{\bd{i}, \bd{k}} S_{\bd{i}}\left[\Phi^{0}_{\bar{t}}(u_1)+\tilde{t} (\Phi^{0}_{\bar{t}}(u_2)-\Phi^{0}_{\bar{t}}(u_1))+v_0\right] \ud \tilde{t} \ud \bar{t}\right) w(0)(\bd{k}). \end{aligned}
\end{equation}
Now choose a $\bd{k}\in\Z^d$ such that $w(0)(\bd{k})>0$ and recall that $\partial_{\bd{i}, \bd{k}} S_{\bd{i}}<0$.
Then from \eqref{eq:dnedddd} it follows that if $\norm{\bd{i}-\bd{k}}=1$, then for all $t > 0$, $w(t)(\bd{i}) > 0$.

To generalize to the case that
$\norm{\bd{i}-\bd{k}}\neq 1$, let us choose a sequence of lattice points
$\bd{j}(0)=\bd{k},\cdots, \bd{j}(N)=\bd{i}$ such that $\left\|\bd{j}(n)-\bd{j}(n-1)\right\|=1$ and $N=\|\bd{i}-\bd{k}\|$.
Then, by induction, $w\left(\frac{n t}{N}\right)(\bd{j}(n))>0$ for any $n\in \{0,\cdots,N\}$.
Thus, if $w(0)(\bd{k})>0$ and $t>0$, then $w(t)(\bd{i})>0$.
\qed

\bigskip

Next we prove Lemma \ref{lem:2.2}.

\noindent
\textbf{\emph{Proof of Lemma \ref{lem:2.2}.}}
\eqref{lem:2.2-1}.
Suppose
$(I^{\bd{p}}_{0})'  (u)=0$. Then $u+v_0$ is a solution of \eqref{eq:PDE}, thus
$\sum_{\bd{j}: \norm{\bd{j}-\bd{i}}\leq r}\partial_{\bd{i}}S_{\bd{j}}(u+v_0)=0$
for all $\bd{i}\in \Z^n$.
So $u$ is a solution of the initial problem:
\begin{equation*}
  \left\{
    \begin{array}{ll}
      -\partial_t \Phi_{t}^{0} (u)(\bd{i}) & =\sum_{\bd{j}: \norm{\bd{j}-\bd{i}}\leq r}\partial_{\bd{i}}S_{\bd{j}}(\Phi_{t}^{0} (u)+v_0), \quad \quad \textrm{for }t>0,\\
      \Phi_{0}^{0} (u)(\bd{i}) & =u(\bd{i}).
    \end{array}
  \right.
\end{equation*}
By the uniqueness of the solution of the above initial problem,
$\Phi_{t}^{0} (u)=u$.

\eqref{lem:2.2-2}.
By \eqref{eq:derivative},
\begin{equation}\label{eq:5.8}
\begin{aligned}
&\frac{\ud}{\ud t} I_{0}^{\bd{p}}\left(\Phi_{t}^{0}(u)\right)\\
=&\sum_{\bd{j}\in\T^{\bd{p}}_{0}} \sum_{\bd{k}:\norm{\bd{k}-\bd{j}}\leq r}\partial_{\bd{k}}S_{\bd{j}}(\Phi_{t}^{0}(u)+v_0)\cdot \partial_t \Phi_{t}^{0} (u)(\bd{k})\\
=&\sum_{\bd{j}\in\T^{\bd{p}}_{0}} \sum_{\bd{k}:\norm{\bd{k}-\bd{j}}\leq r}\partial_{\bd{k}}S_{\bd{j}}(\Phi_{t}^{0}(u)+v_0)\cdot \Big[- \sum_{\bd{l}: \norm{\bd{l}-\bd{k}}\leq r} \partial_{\bd{k}}S_{\bd{l}}(\Phi_{t}^{0} (u)+v_0)\Big]\\
=&-\sum_{\bd{j}\in\T^{\bd{p}}_{0}} \Big[\sum_{\bd{k}:\norm{\bd{k}-\bd{j}}\leq r} \partial_{\bd{j}}S_{\bd{k}}(\Phi_{t}^{0} (u)+v_0)\Big]^2\\
\leq & 0.
\end{aligned}
\end{equation}
Thus \eqref{lem:2.2-2} holds.

\eqref{lem:2.2-3}.
By \eqref{lem:2.2-1}, $\Phi_{t}^{0}(0)=0$ and $\Phi_{t}^{0}(w_0-v_0)=w_0-v_0$.
Thus by Proposition \ref{prop:comparison}, $0=\Phi_{t}^{0}(0)\leq\Phi_{t}^{0}(u)\leq \Phi_{t}^{0}(w_0-v_0)=w_0-v_0$ for any $u\in \MG^{\bd{p}}_{0}$.
So \eqref{lem:2.2-3} follows.

\eqref{lem:2.2-4}.
By \eqref{eq:5.8} and Lemma \ref{prop:2.2369852145},
$I_{0}^{\bd{p}} (\Phi_{t}^{0}(u))$ is non-increasing with a lower bound $c_{0}^{\bd{p}}$.
So $I_{0}^{\bd{p}}\left(\Phi_{t}^{0}(u)\right)$ has a limit as $t\to \infty$.
Since
\begin{equation*}
I_{0}^{\bd{p}}(u)-I_{0}^{\bd{p}}\left(\Phi_{t}^{0}(u)\right)=\int_{0}^{t} \sum_{\bd{j}\in\T^{\bd{p}}_{0}} \Big[\sum_{\bd{k}:\norm{\bd{k}-\bd{j}}\leq r} \partial_{\bd{j}}S_{\bd{k}}(\Phi^{0}_{\tilde{t}} (u)+v_0)\Big]^2 \ud \tilde{t},
\end{equation*}
we have
$$
\sum_{\bd{j}\in\T^{\bd{p}}_{0}} \Big[\sum_{\bd{k}:\norm{\bd{k}-\bd{j}}\leq r} \partial_{\bd{j}}S_{\bd{k}}(\Phi^{0}_{t_i} (u)+v_0)\Big]^2 \to 0
$$
for some sequence $t_i\to\infty$.
Since $\Phi^{0}_{t_i}(u)\in \MG^{\bd{p}}_{0}$, a compact set by Proposition \ref{prop:compact},
there is a $U\in \MG^{\bd{p}}_{0}$
such that $\Phi^{0}_{t_i}(u)\to U$ in $\Lambda_{0}^{\bd{p}}$
along a subsequence
of $t_i$.
Note that $\partial_{\bd{j}}S_{\bd{k}}$ is continuous,
thus
$$
\sum_{\bd{j}\in\T^{\bd{p}}_{0}} \Big[\sum_{\bd{k}:\norm{\bd{k}-\bd{j}}\leq r} \partial_{\bd{j}}S_{\bd{k}}(U+v_0)\Big]^2 = 0,
$$
which implies
$$
\sum_{\bd{k}:\norm{\bd{k}-\bd{i}}\leq r} \partial_{\bd{i}}S_{\bd{k}}(U+v_0) = 0,\quad \textrm{for  }\bd{i}\in\Z^n,
$$
i.e., $U+v_0$ is a solution of \eqref{eq:PDE}.

\eqref{lem:2.2-5}.
Firstly we claim that: there exists a constant $\epsilon>0$ such that if $K_c=\emptyset$, then
\begin{equation}\label{eq:5.10}
\left\|(I_{0}^{\bd{p}})'(u)\right\|_{(\Lambda_{0}^{\bd{p}})'} \geq \sqrt{2\epsilon}, \quad \textrm{for  }u \in (I_{0}^{\bd{p}})_{c-\epsilon}^{c+\epsilon}.
\end{equation}
Here $(\Lambda_{0}^{\bd{p}})'$ is the dual space of the Banach space $\Lambda_{0}^{\bd{p}}$ and $(I_{0}^{\bd{p}})^{t_1}_{t_2}:=\{u\in \MG_{0}^\bd{p}\,|\, t_2\leq I_{0}^{\bd{p}}(u)\leq t_1\}$.
Indeed, if the claim fails then for any $k\in\N$,
there are $(u_k)\subset (I_{0}^{\bd{p}})_{c-1/k}^{c+1/k} $ satisfying
$$
\left\|(I_{0}^{\bd{p}})'(u_k)\right\|_{(\Lambda_{0}^{\bd{p}})'}\to 0
$$
as $k\to\infty$.
But since $(u_k)\subset \MG^{\bd{p}}_{0}$, a compact set by Proposition \ref{prop:compact}, we may assume that $u_k\to U\in \MG^{\bd{p}}_{0}$ as $k\to \infty$ and $I_{0}^{\bd{p}}(U)=c$.
Thus $U\in K_c$, which contradicts $K_c= \emptyset$.
So \eqref{eq:5.10} holds.

Now we prove \eqref{lem:2.2-5}.
Suppose, by contradiction, that
there is a $u\in (I_{0}^{\bd{p}})_{c-\epsilon}^{c+\epsilon}$ such that $\Phi_{1}^{0}(u)\not\in (I_{0}^{\bd{p}})^{c-\epsilon}$,
then $c-\epsilon< I_{0}^{\bd{p}}(\Phi_{t}^{0}(u)) \leq c+\epsilon $ for all $t\in [0,1]$.
Since
\begin{equation}\label{eq:lllll}
\begin{split}
&\left\|(I_{0}^{\bd{p}})'(u)\right\|_{(\Lambda_{0}^{\bd{p}})'}\\
=&\sup_{\norm{v}_{\Lambda_{0}^{\bd{p}}}\leq 1}|(I_{0}^{\bd{p}})'(u)v|\\
=&\sup_{\norm{v}_{\Lambda_{0}^{\bd{p}}}\leq 1}\Big|\sum_{\bd{j}\in \T^{\bd{p}}_{0}}\sum_{\bd{k}:\norm{\bd{k}-\bd{j}}\leq r} \partial_{\bd{k}}S_{\bd{j}}(u+v_0)v(\bd{k})\Big|\\
\leq &\sup_{\norm{v}_{\Lambda_{0}^{\bd{p}}}\leq 1}\sum_{\bd{j}\in\T^{\bd{p}}_{0}}\Big|v(\bd{j})\cdot\sum_{\bd{k}:\norm{\bd{k}-\bd{j}}\leq r}\partial_{\bd{j}}S_{\bd{k}}(u+v_0)\Big|\\
\leq &\sup_{\norm{v}_{\Lambda_{0}^{\bd{p}}}\leq 1}\left(\sum_{\bd{j}\in \T^{\bd{p}}_{0}}\left|v(\bd{j})\right|^2\right)^{\frac{1}{2}}\left(\sum_{\bd{j}\in \T^{\bd{p}}_{0}}\left|\sum_{\bd{k}:\norm{\bd{k}-\bd{j}}\leq r}\partial_{\bd{j}}S_{\bd{k}}(u+v_0)\right|^2\right)^{\frac{1}{2}}\\
=&\left(\sum_{\bd{j}\in \T^{\bd{p}}_{0}}\left|\sum_{\bd{k}:\norm{\bd{k}-\bd{j}}\leq r}\partial_{\bd{j}}S_{\bd{k}}(u+v_0)\right|^2\right)^{\frac{1}{2}},
\end{split}
\end{equation}
we have
\begin{equation*}
  \begin{split}
    &I_{0}^{\bd{p}}(\Phi_{1}^{0}(u)) \\
    =& I_{0}^{\bd{p}}(u)+\int_{0}^{1} \frac{\ud}{\ud t}I_{0}^{\bd{p}}(\Phi_{t}^{0}(u))\ud t\\
      =& I_{0}^{\bd{p}}(u)+\int_{0}^{1} (I_{0}^{\bd{p}})' (\Phi_{t}^{0}(u))\partial_{t}\Phi_{t}^{0}(u) \ud t\\
      =&I_{0}^{\bd{p}}(u)-\int_{0}^{1}\sum_{\bd{j}\in\T^{\bd{p}}_{0}} \Big[\sum_{\bd{k}:\norm{\bd{k}-\bd{j}}\leq r} \partial_{\bd{j}}S_{\bd{k}}(\Phi_{t}^{0} (u)+v_0)\Big]^2\ud t\\
      \leq &(c+\epsilon) - 2\epsilon\\
      =&c-\epsilon,
  \end{split}
\end{equation*}
contrary to the existence of $u$.
Here the third equality follows from \eqref{eq:5.8} and the first inequality follows from \eqref{eq:5.10}-\eqref{eq:lllll}.
\qed

\bigskip
Now a heat flow method is used to give the following (cf. \cite[Proposition 2.12]{Rabi2014}):

\noindent
\textbf{\emph{Another proof of Theorem \ref{thm:2.1}.}}
For $\epsilon>0$, let $h\in C([0,1], \MH_{0}^{\bd{p}}\cap (I_{0}^{\bd{p}})^{d_{0}^{\bd{p}}+\epsilon})$.
Set $h_t:=\Phi_{t}^{0} \circ h$, then for any $t\geq 0$,
$h_t\in C([0,1], \MH_{0}^{\bd{p}}\cap (I_{0}^{\bd{p}})^{d_{0}^{\bd{p}}+\epsilon})$.
We claim that:
\begin{equation*}
  \textrm{there is a $\theta_{\infty} \in [0,1]$ satisfying $I_{0}^{\bd{p}}\left(h_{t}\left(\theta_{\infty}\right)\right) \geq d_{0}^{\bd{p}}$ for all $t\geq 0$.}
\end{equation*}
Indeed, for any $t\geq 0$ there exists $\theta_{t} \in [0,1]$ such that $I_{0}^{\bd{p}}\left(h_{t}\left(\theta_{t}\right)\right) \geq d_{0}^{\bd{p}}$.
We can extract a subsequence of $\theta_t$, say $\theta_{t_{k}}$ converging to some $\theta_{\infty} \in [0,1]$ as $t_k\to \infty$.
If $I_{0}^{\bd{p}}\left(h_{\tau}\left(\theta_{\infty}\right)\right)<d_{0}^{\bd{p}}$ for some $\tau>0$,
then
$I_{0}^{\bd{p}}\left(h_{\tau}\left(\theta_{t_{k}}\right)\right)<d_{0}^{\bd{p}}$ for large $t_k$.
Then enlarging $t_k$ if necessary
such that $t_{k}>\tau$,
by Lemma \ref{lem:2.2} \eqref{lem:2.2-2},
$I_{0}^{\bd{p}}\left(h_{t_{k}}\left(\theta_{t_{k}}\right)\right)\leq I_{0}^{\bd{p}}\left(h_{\tau}\left(\theta_{t_{k}}\right)\right)<d_{0}^{\bd{p}}$, which is a contradiction.
Thus the claim holds.

By Lemma \ref{lem:2.2} \eqref{lem:2.2-4},
there is a subsequence of $h_t(\theta_{\infty})$ converging to some $v_{\epsilon}\in \MG_{0}^{\bd{p}}$ such that
$v_{\epsilon}+v_0$ is a solution of \eqref{eq:PDE} and
\begin{equation*}
d_{0}^{\bd{p}}+\epsilon \geq I_{0}^{\bd{p}}\left(v_{\epsilon}\right)=\lim _{t \rightarrow \infty} I_{0}^{\bd{p}}\left(h_{t}\left(\theta_{\infty}\right)\right) \geq d_{0}^{\bd{p}}.
\end{equation*}
By Proposition \ref{prop:compact}, letting $\epsilon \to 0$ (up to a subsequence) completes the proof of Theorem \ref{thm:2.1}.
\qed

\subsection*{Acknowledgments}
The author wishes to express his gratitude to Professor Zhi-Qiang Wang (Utah State University)
for helpful discussion and for giving me a lot of encouragement.
The author is supported by the Fundamental Research Funds for the Central Universities (no. 34000-31610274).


\end{CJK*}
\end{document}